\documentclass[10pt]{amsart}
\usepackage{amsmath,amscd}
\usepackage{amssymb}

\usepackage[all]{xy}

\usepackage[T1]{fontenc}

\usepackage{times}

\title{Derived Hochschild functors over commutative adic algebras} 

\author{Liran Shaul}
\address{Department of  Mathematics 
The Weizmann Institute of Science, 
Rehovot 76100, 
Israel}
\email{liran.shaul@weizmann.ac.il}
\newtheorem{thm}[equation]{Theorem}
\newtheorem{cor}[equation]{Corollary}
\newtheorem{prop}[equation]{Proposition}
\newtheorem{lem}[equation]{Lemma}
\theoremstyle{definition}
\newtheorem{dfn}[equation]{Definition}
\newtheorem{rem}[equation]{Remark}
\newtheorem{exa}[equation]{Example}

\numberwithin{equation}{section}

\newcommand{\inj}{\hookrightarrow}

\newcommand{\xar}{\xrightarrow}

\newcommand{\opn}{\operatorname}
\newcommand{\cat}[1]{\operatorname{\mathsf{#1}}}

\newcommand{\mfrak}[1]{\mathfrak{#1}}

\newcommand{\mrm}[1]{\mathrm{#1}}

\newcommand{\tup}[1]{\textup{#1}}
\newcommand{\bsym}[1]{\boldsymbol{#1}}

\renewcommand{\k}{\Bbbk}

\renewcommand{\d}{\mathrm{d}}

\newcommand{\mf}[1]{\mathfrak{#1}}
\renewcommand{\a}{\mfrak{a}}
\renewcommand{\b}{\mfrak{b}}
\renewcommand{\c}{\mfrak{c}}

\begin{document}

\begin{abstract}
Let $\k$ be a commutative ring, and let $(A,\mfrak{a})$ be an adic ring which is a $\k$-algebra. We study complete and torsion versions of the derived Hochschild homology and cohomology functors of $A$ over $\k$. To do this, we first establish weak proregularity of certain ideals in flat base changes of noetherian rings. Next, we develop a theory of DG-affine formal schemes, extending the Greenlees-May duality and the MGM equivalence to this setting. Finally, we define complete and torsion derived Hochschild homology and cohomology functors in this setting, 
and show that if $\k$ is noetherian and $(A,\mfrak{a})$ is essentially of finite type (in the adic sense) over $\k$, then there are formulas to compute them that stay inside the noetherian category. In the classical case, where $\k$ is a field, we deduce that topological Hochschild cohomology and discrete Hochschild cohomology are isomorphic.
\end{abstract}

\maketitle

\setcounter{section}{-1}
\section{Introduction}

All rings in this paper are assumed to be commutative. Let $\k$ be a commutative ring, let $A$ be a commutative noetherian $\k$-algebra, and let $\a \subseteq A$ be an ideal. Assume that $A$ is $\a$-adically complete. In this paper we study adic variants of the Hochschild homology and cohomology functors of $A$ over $\k$.

For a ring $A$, we denote by $\opn{Mod} A$ the category of $A$-modules. We denote by $\mrm{D}(\opn{Mod} A)$ the derived category of $A$-modules. We shall in general follow the notations of \cite{RD} concerning derived categories. We denote by $\Gamma_{\a} (-)$ and $\Lambda_{\a}(-)$ the $\a$-torsion and $\a$-completion functors respectively, and by $\mrm{R}\Gamma_{\a}(-)$ and $\mrm{L}\Lambda_{\a}(-)$ their derived functors.

If $\k$ is a field, then the Hochschild homology and cohomology of an $A\otimes_{\k} A$-module $M$ of $A$ over $\k$ may be defined as the cohomologies of the complexes
\[
A\otimes^{\mrm{L}}_{A\otimes_{\k} A} M
\]
and
\[
\mrm{R}\opn{Hom}_{A\otimes_{\k} A}(A, M) 
\]
respectively.

When $\k$ is not necessarily a field, and the map $\k \to A$ might fail to be flat, an important generalization of the second construction above is a variant of the Shukla cohomology, known as the derived Hochschild cohomology. This functor, recently studied in \cite{AILN}, may be defined as the functor $\mrm{D}(\opn{Mod} A) \times \mrm{D}(\opn{Mod} A) \to \mrm{D}(\opn{Mod} A)$ given by
\begin{equation}\label{eqn-derived-hoc}
\mrm{R}\opn{Hom}_{A\otimes^{\mrm{L}}_{\k} A}(A,M\otimes^{\mrm{L}}_{\k} N).
\end{equation}
for any $M,N \in \mrm{D}(\opn{Mod} A)$.
Here, $A\otimes^{\mrm{L}}_{\k} A$ is the derived tensor product of $A$ with itself over $\k$ in the category of (super-commutative) differential graded algebras. All necessary fact about differential graded algebras used in this paper are recalled in Section 2 below. An important special case of this functor where $M=N$ was recently studied in \cite{YZ1,YZ1E}. In a more general context of ind-coherent sheaves over a DG-scheme, this functor was recently studied in \cite{GA}, where the notation $M \overset{!}\otimes N$ was used for it.

Even in the simplest case where $\k$ is a field, there is a major challenge when working with these formulas in the adic category. The problem is that usually the map $\k \to A$ is not of finite type, and hence, in many cases of interest, the ring $A\otimes_{\k} A$ is not noetherian. This is the case for example if $\k$ is a field of characteristic $0$, and $A=\k[[t]]$ is the ring of formal power series over $\k$. 

The ring $A\otimes_{\k} A$ also carries a natural topology on it, namely, the $\a^e = \a \otimes_{\k} A + A\otimes_{\k} \a$-adic topology. In most interesting cases, the completion $\widehat{A\otimes_{\k} A}$ is a noetherian ring. Thus, it is natural to replace the ordinary Hochschild cohomology with the functor
\begin{equation}\label{eqn-adic-derived-hoc}
\mrm{R}\opn{Hom}_{\widehat{A\otimes_{\k} A}}(A, K)
\end{equation}
where $K \in \mrm{D}(\opn{Mod} (\widehat{A\otimes_{\k} A}) )$.

Given $M,N \in \mrm{D}(\opn{Mod} A)$, we would like to imitate equation (\ref{eqn-derived-hoc}) above, and replace $K$ in equation (\ref{eqn-adic-derived-hoc}) with an image of $M\otimes^{\mrm{L}}_{\k} N$ in $\mrm{D}(\opn{Mod} (\widehat{A\otimes_{\k} A}) )$. There are two natural homological ways to do this, one is to consider the derived completion $\mrm{L}\Lambda_{\a^e}(M\otimes^{\mrm{L}}_{\k} N) \in \mrm{D}(\opn{Mod} (\widehat{A\otimes_{\k} A}) )$, and the other is to consider the derived torsion 
$\mrm{R}\Gamma_{\a^e}(M\otimes^{\mrm{L}}_{\k} N) \in \mrm{D}(\opn{Mod} (\widehat{A\otimes_{\k} A}) )$. 

But now, another difficulty arises, namely, that we must work with the derived completion and derived torsion functors over the (possibly) non-noetherian ring $A\otimes_{\k} A$. 

Over a noetherian ring, there is a well known explicit formula for computing the derived torsion $\mrm{R}\Gamma_{\a} (M)$, given by tensoring $M$ with the infinite dual Koszul complex of the ideal $\a$.  A similar result is true for the derived completion $\mrm{L}\Lambda_{\a}$. Over a non-noetherian ring, such results are not always true. Recently, it was determined when such results are true, and this is the content of the Greenlees-May theory, initiated in \cite{GM2}, and developed in \cite{AJL,SC,PSY1}. 

According to this theory (which is recalled in details in Section 1 below), given a ring $A$, and a finitely generated ideal $\a \subseteq A$, there is a technical notion called a weakly proregular ideal, which $\a$ may or may not satisfy (More precisely, it is a property of the adic topology on $A$). This theory shows that the derived completion and derived torsion functors have good behavior (for instance, finite cohomological dimension, and the explicit formulas which were valid over noetherian rings), exactly when $\a$ is weakly proregular. In particular, in a noetherian ring, every ideal is weakly proregular. 

Thus, the first main result of this paper is that the ideal $\a^e$ discussed above is weakly proregular. More precisely:

\begin{thm}\label{thm-ind-wpr}
Let $\k$ be a noetherian ring. Let $(A,\a)$ be an adic ring which is a flat  essentially formally of finite type $\k$-algebra. Then the ideal $\a^e = \a \otimes_{\k} A + A\otimes_{\k} \a \subseteq A\otimes_{\k} A$ is weakly proregular.
\end{thm}

This is contained in Corollary \ref{cor-wpr-of-efft} in the body of the paper. Here, being essentially formally of finite type means that $A/\a$ is a localization of a finitely generated $\k$-algebra.

If one is only interested in working over a base field, then this result is sufficient to develop the theory of adic Hochschild homology and cohomology developed in Sections 3 and 4 below. To treat the derived Hochschild homology and cohomology theories in the adic category, we must first develop a theory of adic DG-algebras, and this is done in Section 2 below.

Let $A$ be a super-commutative non-positive DG-algebra. We denote by $\widetilde{\mrm{D}}(\opn{DGMod} A)$ the derived category of DG-modules over $A$.

We introduce a technical notion, called an internally flat DG-algebra (See Definition \ref{dfn-internally-flat}). We show that any DG-algebra is quasi-isomorphic to an internally flat DG-algebra (Proposition \ref{prop-existence-of-res}).

Given a DG-algebra $A$, and a finitely generated DG-ideal $\a \subseteq A$ which is generated by degree $0$ elements,  we extend the $\a$-torsion and $\a$-completion functors to the category of DG-modules, 

Under the internal flatness assumption, we extend the Greenlees-May theory to the setting of DG-algebras. The results of Section 2 are summarized in the following theorem:

\begin{thm}
Let $A$ be an internally flat DG-algebra, and let $\a \subseteq A$ be a finitely generated DG-ideal, generated by degree $0$ elements. 
Let $\mathbf{a}\subseteq A$ be a finite sequence of degree $0$ elements that generate $\a$. Assume that the ideal generated by $\mathbf{a}$ in $A^0$ is weakly proregular. Then the following holds:
\begin{enumerate}
\item There are isomorphisms
\[
\mrm{R}\Gamma_{\a} (-) \cong \opn{K}^{\vee}_{\infty}(A;\mathbf{a})\otimes_A -
\]
and
\[
\mrm{L}\Lambda_{\a} (-) \cong \opn{Hom}_A(\opn{Tel}(A;\mathbf{a}), -)
\]
of functors $\widetilde{\mrm{D}}(\opn{DGMod} A) \to \widetilde{\mrm{D}}(\opn{DGMod} A)$.
\item \textbf{MGM equivalence}: For any $M\in \widetilde{\mrm{D}}(\opn{DGMod} A)$ the functorial maps
\[
\mrm{R}\Gamma_{\a} (\tau^L_M) : \mrm{R}\Gamma_{\a}(M) \to \mrm{R}\Gamma_{\a}(\mrm{L}\Lambda_{\a} (M) )
\]
and
\[
\mrm{L}\Lambda_{\a} (\sigma^R_M) : \mrm{L}\Lambda_{\a} (\mrm{R}\Gamma_{\a}(M)) \to \mrm{L}\Lambda_{\a} (M)
\]
are isomorphisms.
\item \textbf{Greenlees-May duality}: For any $M,N\in \widetilde{\mrm{D}}(\opn{DGMod} A)$ the morphisms
\[ \begin{aligned}
& \mrm{R}\opn{Hom}_A \bigl( \mrm{R} \Gamma_{\a} (M), \mrm{R} \Gamma_{\a} (N) \bigr)
\xar{\mrm{R}\opn{Hom}(1, \sigma^{\mrm{R}}_N)}
\mrm{R}\opn{Hom}_A \bigl( \mrm{R} \Gamma_{\a} (M), N \bigr) 
\\
& \qquad\xar{\mrm{R}\opn{Hom}(1, \tau^{\mrm{L}}_N)}
\mrm{R}\opn{Hom}_A \bigl( \mrm{R} \Gamma_{\a} (M), \mrm{L} \Lambda_{\a} (N) \bigr) 
\xleftarrow{\mrm{R}\opn{Hom}(\sigma^{\mrm{R}}_M, 1)}
\\ & \qquad  \qquad
\mrm{R}\opn{Hom}_A \bigl( M, \mrm{L} \Lambda_{\a} (N) \bigr)
\xleftarrow{\mrm{R}\opn{Hom}(\tau^{\mrm{L}}_M, 1)}
\mrm{R}\opn{Hom}_A \bigl( \mrm{L} \Lambda_{\a} (M), \mrm{L} \Lambda_{\a} (N) \bigr)
\end{aligned} \]
in $\widetilde{\mrm{D}}(\opn{DGMod} A)$ are isomorphisms.
\end{enumerate}
\end{thm}
This theorem is contained in Theorems \ref{thm-rgamma}, \ref{thm-llambda}, \ref{thm-RgammaLLambda}, \ref{thm-LLambdaRGamma} and Corollary \ref{cor-DG-GM} in the body of the paper.

Here, the DG-modules $\opn{K}^{\vee}_{\infty}(A;\mathbf{a})$ 
and $\opn{Tel}(A;\mathbf{a})$ are the infinite dual Koszul and Telescope DG-modules respectively, introduced in Remark \ref{rem-koszul-dg} below.

With this theorem in hand, we study in Section 3 the adic derived Hochschild cohomology functor. The main result of section 3 is:

\begin{thm}
Let $\k$ be a commutative ring. Let $A$ be a $\k$-algebra, and let $\a \subseteq A$ be a finitely generated ideal. Assume that $\a$ is weakly proregular, and that $A$ is $\a$-adically complete. Let $\k \to \widetilde{A} \to A$ be a flat DG-resolution of $\k \to A$ which is internally flat. Let $I\subseteq (\widetilde{A}\otimes_{\k} \widetilde{A})^0$ be a weakly proregular ideal, such that its image under the 
composed map 
\[
(\widetilde{A}\otimes_{\k} \widetilde{A})^0 \to \widetilde{A}^0 \to A
\]
is equal to $\a$. Then there is an isomorphism 
\[
\mrm{L}\Lambda_{\a} (\mrm{R}\opn{Hom}_{A\otimes^{\mrm{L}}_{\k} A} ( A, -\otimes^{\mrm{L}}_{\k} -) )  
\cong 
\mrm{R}\opn{Hom}_{\Lambda_I(\widetilde{A}\otimes_{\k} \widetilde{A})} (A, \mrm{L}\Lambda_I ( - \otimes^{\mrm{L}}_{\k} - ) )
\]
of functors
\[
\mrm{D}(\opn{Mod} A) \times \mrm{D}(\opn{Mod} A) \to \mrm{D}(\opn{Mod} A)_{\opn{\a-com}}
\]
\end{thm}
This is repeated as Theorem \ref{thm-main-c-formula} in the body of the paper. In the noetherian case, and using Theorem \ref{thm-ind-wpr}, we may refine this, and get:
\begin{cor}
Let $\k$ be a noetherian ring. Let $A$ be a $\k$-algebra, and let $\a \subseteq A$ be a finitely generated ideal. Assume that $A$ is $\a$-adically complete and essentially formally of finite type over $\k$.  Then for any adic flat DG-resolution $\k \to (\widetilde{A},\widetilde{\a}) \to (A,\a)$ of $\k \to (A,\a)$, such that $(\widetilde{A}^0,\widetilde{\a})$ is essentially formally of finite type over $\k$, 
there is an isomorphism 
\[
\mrm{L}\Lambda_{\a} (\mrm{R}\opn{Hom}_{A\otimes^{\mrm{L}}_{\k} A} ( A, -\otimes^{\mrm{L}}_{\k} -) )  
\cong 
\mrm{R}\opn{Hom}_{\Lambda_{\widetilde{\a}^e} (\widetilde{A}\otimes_{\k} \widetilde{A})} (A, \mrm{L}\Lambda_{\widetilde{\a}^e} ( - \otimes^{\mrm{L}}_{\k} - ) )
\]
of functors
\[
\mrm{D}(\opn{Mod} A) \times \mrm{D}(\opn{Mod} A) \to \mrm{D}(\opn{Mod} A)_{\opn{\a-com}}
\]
where $\widetilde{\a}^e = \widetilde{\a} \otimes_{\k} \widetilde{A}^0 + \widetilde{A}^0 \otimes_{\k} \widetilde{\a}$ is the canonical ideal of definition of $\widetilde{A}^0 \otimes_{\k} \widetilde{A}^0$. Moreover, the ring $\Lambda_{\widetilde{\a}^e} (\widetilde{A}\otimes_{\k} \widetilde{A})^0$ is noetherian. Furthermore, such a resolution $(\widetilde{A},\widetilde{\a})$ always exist. If $A$ is flat over $\k$, then one can take $\widetilde{A} = A$ and  $\widetilde{\a} = \a$.
\end{cor}
This is repeated as Corollary \ref{cor-cformula-noetherian} in the body of the paper. Here, the left hand side is the derived completion of the (discrete) derived Hochschild cohomology, while the right hand side is an adic analogue of the derived Hochschild cohomology. 

A similar result under more restrictive assumption is given in the torsion case in Theorem \ref{thm-main-t-formula} below.

Restricting further to the case where $\k$ is a field, we obtain the following comparison theorem: 

\begin{cor}
Let $\k$ be a field. Let $A$ be a $\k$-algebra. Assume that $A$ is noetherian and $\a$-adically complete with respect to some ideal $\a\subseteq A$. Let $M$ be a finitely generated $A$-module. Then for all $n$, there is an isomorphism
\[
\mrm{HH}^n(A|\k;M) \cong \widehat{\mrm{HH}}^n(A|\k;M)
\]
between the $n$-th Hochschild cohomology of $A$ over $\k$ with coefficients in $M$, and the $n$-th topological Hochschild cohomology of $A$ over $\k$ with coefficients in $M$.
\end{cor}
This is repeated as Corollary \ref{cor-comp-thm} in the body of the paper.

In the short and final Section 4, we give a similar result in the torsion case for derived Hochschild homology. The result is:
\begin{thm}
Let $\k$ be a commutative ring. Let $A$ be a $\k$-algebra, and let $\a \subseteq A$ be a finitely generated ideal. Assume that $\a$ is weakly proregular, and that $A$ is $\a$-adically complete. Let $\k \to \widetilde{A} \to A$ be a flat DG-resolution of $\k \to A$ which is internally flat. Let $I\subseteq (\widetilde{A}\otimes_{\k} \widetilde{A})^0$ be a weakly proregular ideal, such that its image under the composed map 
\[
(\widetilde{A}\otimes_{\k} \widetilde{A})^0 \to \widetilde{A}^0 \to A
\]
is equal to $\a$. Then there is an isomorphism 
\[
\mrm{R}\Gamma_{\a} (A \otimes^{\mrm{L}}_{A\otimes^{\mrm{L}}_{\k} A} \mrm{R}\opn{Hom}_{\k}(-,-)) \cong 
A \otimes^{\mrm{L}}_{\Lambda_I(\widetilde{A} \otimes_{\k} \widetilde{A})} \mrm{R}\Gamma_I(\mrm{R}\opn{Hom}_{\k}(-,-))
\]
of functors
\[
\mrm{D}(\opn{Mod} A) \times \mrm{D}(\opn{Mod} A) \to \mrm{D}(\opn{Mod} A)_{\opn{\a-tor}}.
\]
\end{thm}
This is repeated as Theorem \ref{thm-h-t-formula} in the body of the paper.

We end the introduction with mentioning our motivation for studying these functors: in the subsequent article \cite{SH}, we use the derived adic Hochschild cohomology theory developed here to develop a theory of rigid dualizing complexes over affine formal schemes, generalizing results of \cite{AIL,YZ1,YZ3} to the formal setting.

\textbf{Acknowledgments.}
The author would like to thank Amnon Yekutieli for many helpful suggestions and discussions.

\section{Weakly proregular preadic rings}

In this section we introduce and study the category of weakly proregular preadic rings. We do not assume rings are necessarily noetherian. 

Let $A$ be a ring. Given an ideal $\mfrak{a}\subseteq A$, recall that the $\mfrak{a}$-adic topology on $A$ is the linear topology generated by powers of $\mfrak{a}$.

\begin{dfn}\label{preadicdef}
The category of preadic rings, denoted by $\mathcal{PA}$, is defined as follows:
\begin{enumerate}
\item 
A \textbf{preadic ring}, is a commutative ring $A$ endowed with a topology, which is the $\mfrak{a}$-adic topology for some finitely generated ideal $\mfrak{a}\subseteq A$. (In \cite{EGA1}, Definition 7.1.9 this terminology is introduced without the finiteness condition). 
\item A presentation of a preadic ring $A$ is a pair $(A,\mfrak{a})$, where $\mfrak{a}$ is some finitely generated defining ideal of $A$. By abuse of notation, we shall say that $(A,\mfrak{a})$ is a preadic ring.
\item Given two preadic rings $(A,\mfrak{a})$, and $(B,\mfrak{b})$, a preadic morphism $f:A\to B$ is a ring homomorphism which is a continuous map with respect to the adic topologies. In this situation, we shall say that $(B,\mfrak{b})$ is a preadic $(A,\mfrak{a})$-algebra. This is equivalent to the fact that $\mfrak{a}^j \subseteq f^{-1}(\mfrak{b})$ for some $j \in \mathbb{N}$.
\end{enumerate}
\end{dfn}

\begin{rem}\label{remark-presentation}
Two presentations $(A,\mfrak{a})$ and $(A,\mfrak{a}')$ represent the same preadic ring if and only if $\sqrt{\mfrak{a}} = \sqrt{\mfrak{a}'}$.
\end{rem}

\begin{rem}
All rings in this paper are assumed to be preadic. Thus, if we say that $A$ is a ring, we will mean that $A$ is discrete, that is, that $(A,0)$ is a preadic ring.
\end{rem}

\begin{rem}
Finiteness conditions on preadic algebra maps are defined as follows: If $(A,\a)$ is a preadic ring, and if $(B,\b)$ is a preadic $(A,\a)$-algebra, then $B$ is (essentially) formally of finite type over $A$ if $B/\b$ is (essentially) of finite type over $A$. (Recall that being essentially of finite type over $A$ means being a quotient of a localization of a polynomial ring in finitely many variables over $A$).
This is the terminology of \cite{YE2}. In \cite{LNS}, such maps are called essentially pseudo-finite-type maps. 
\end{rem}

Given a preadic ring $(A,\mfrak{a})$, there are two functors associated to this data: the $\mfrak{a}$-torsion functor and the $\mfrak{a}$-completion functor, $\Gamma_{\a},\Lambda_{\a} :\cat{Mod}(A)\to \cat{Mod}(A)$, defined by
\[
\Gamma_{\a} M := \varinjlim \opn{Hom}_A(A/\a^n,M)
\]
and
\[
\Lambda_{\a} M := \varprojlim A/\a^n \otimes_A M.
\]
Note that by Remark \ref{remark-presentation}, these definitions are independent of the chosen presentation of $(A,\mfrak{a})$.

Let $(A,\mfrak{a})$ be a preadic ring. For any $A$-module $M$, there are canonical maps $\sigma_M:\Gamma_{\a} (M) \to M$, and $\tau_M:M\to \Lambda_{\a} (M)$. If these maps are bijective then $M$ is said to be $\a$-torsion and $\a$-adically complete respectively. The $A$-module $\Lambda_{\a} (A)$ has a structure of a commutative ring, and (since $\a$ is finitely generated), it is $\widehat{\a} = \a \Lambda_{\a}(A)$-adically complete. In general, a preadic ring $(A,\a)$ is called an adic ring if $A$ is $\a$-adically complete. The category of adic rings is a full-subcategory of the category of preadic rings. If $(A,\mfrak{a})$ is a preadic ring, then $\Lambda_{\a} (A)$ is a noetherian ring if and only if $A/\a$ is a noetherian ring. (\cite{CA}, Corollary 2 after
Proposition III.2.11.14). It follows that if $A$ is a noetherian ring, and if $(B,\b)$ is an adic ring which is an essentially formally of finite type $A$-algebra, then $B$ is also a noetherian ring.

Next, we recall the homological properties of the torsion and completion functors. Let $(A,\a)$ be a preadic ring. The torsion functor $\Gamma_{\a}$ is left exact. The completion functor $\Lambda_{\a}$ preserves surjections, and in general is neither left exact nor right exact. If $A$ is noetherian then $\Lambda_{\a}$ is exact on the category of finitely generated $A$-modules, and the map $A\to \Lambda_{\a}(A)$ is flat.

The derived functors $\mrm{R}\Gamma_{\a}, \mrm{L}\Lambda_{\a}:\mrm{D}(\opn{Mod} A)\to \mrm{D}(\opn{Mod} A)$ exist. They are calculated using K-injective and K-flat resolutions. (\cite{AJL}). 

For any complex $M$, there are canonical maps $\sigma^R_M:\mrm{R}\Gamma_{\a} M \to M$ and $\tau^L_M:M \to \mrm{L}\Lambda_{\a} M$. (\cite{PSY1}, Proposition 2.7 and Proposition 2.10). If these maps are isomorphisms then $M$ is said to be cohomologically $\a$-torsion or cohomologically $\a$-adically complete respectively. The full subcategories of $\mrm{D}(\opn{Mod} A)$ consisting of cohomologically $\a$-torsion and cohomologically $\a$-adically complete complexes are denoted by $\mrm{D}(\opn{Mod} A)_{\opn{\a-tor}}$ and $\mrm{D}(\opn{Mod} A)_{\opn{\a-com}}$ respectively. They are both triangulated subcategories.

\par We now discuss weak proregularity. 
First, recall that given a ring $A$, and a finite sequence of elements $\mathbf{a} = (a_1,\dots,a_n)$ in $A$, there is a bounded complex of countably generated free $A$-modules denoted by $\opn{Tel}(A;\mathbf{a})$, called the telescope complex, defined as follows (See \cite{PSY1} Section 4 for more details): First, for each element $a\in A$, let $\opn{F}_{\mrm{fin}}(\mathbb{N},A)$ be the free $A$-module of countable rank with basis $\delta_0,\delta_1,\dots$, and let
\[ \opn{Tel}(A; a) :=
\bigl( \cdots \to 0 \to \opn{F}_{\mrm{fin}}( \mathbb{N}, A) \xar{\d}
\opn{F}_{\mrm{fin}}( \mathbb{N}, A) \to 0 \to \cdots \bigr) \]
be the complex concentrated in degrees $0$ and $1$ with differential
\[ \d(\delta_i) := 
\begin{cases}
\delta_0  & \tup{ if } i = 0 , \\
\delta_{i- 1} - a \delta_{i} & \tup{ if } i \geq 1 .
\end{cases} \]
In general, if $\mathbf{a} = (a_1,\dots,a_n)$, we set
\[
\opn{Tel}(A;\mathbf{a}) :=  \opn{Tel}(A;a_1)\otimes_A \opn{Tel}(A;a_2) \otimes_A \dots \otimes_A \opn{Tel}(A;a_n)
\]
If $\mathbf{a}$ and $\mathbf{b}$ are two finite sequences which generate the ideals $\mfrak{a}$ and $\mfrak{b}$ respectively, and if $\sqrt{\mfrak{a}} = \sqrt{\mfrak{b}}$, then the two complexes $\opn{Tel}(A;\mathbf{a})$ and $\opn{Tel}(A;\mathbf{b})$ are homotopically equivalent (\cite{PSY1}, Theorem 5.1). If $f:A\to B$ is a ring homomorphism, $\mathbf{a}$ is a finite sequence of elements in $A$, and $\mathbf{b}$ is the image of $\mathbf{a}$ under $f$, then there is a canonical isomorphism of complexes
\[
\opn{Tel}(A;\mathbf{a}) \otimes_A B \cong \opn{Tel}(B;\mathbf{b})
\]

\begin{dfn}\label{wpr-def}
Let $A$ be a ring. 
\begin{enumerate}
\item A finite sequence $\mathbf{a}$ of elements of $A$ is said to be weakly proregular if for every injective $A$-module $I$, and for every $k\ne 0$, we have that
\[
H^k(\opn{Tel}(A;\mathbf{a}) \otimes_A I) = 0.
\]
\item An ideal $\mfrak{a}\subseteq A$ is said to be weakly proregular if there is a finite sequence $\mathbf{a}$ of $A$ which generates $\mfrak{a}$, such that $\mathbf{a}$ is weakly proregular.
\end{enumerate}
\end{dfn}

\begin{rem}
This definition was first given in \cite{AJLC}, following \cite{AJL}, Lemma 3.1.1. For the history and the etymology of this concept, we refer the reader to \cite{AJL} and \cite{SC}.
\end{rem}

\begin{rem}
If $\mathbf{a}$ and $\mathbf{a}'$ are two finite sequences which generate ideals with equal radicals, then $\mathbf{a}$ is weakly proregular if and only if $\mathbf{a}'$ is. (\cite{SC}, Proposition 1.2). In other words, weak proregularity is a property of adic topologies rather then finite sequences or ideals.
\end{rem}

\begin{rem}
If $A$ is a noetherian ring, then every ideal $\a \subseteq A$ is weakly proregular. (\cite{AJL}).
\end{rem}

The next theorem explains the importance of weak proregularity:

\begin{thm}\label{thm-ring-mgm}
Let $A$ be a ring. Let $\mfrak{a}\subseteq A$ be a weakly proregular ideal, and let $\mathbf{a}$ be a finite sequence of elements of $A$ which generates $A$.
\begin{enumerate}
\item For any $M\in \mrm{D}(\opn{Mod} A)$ there are functorial isomorphisms in $\mrm{D}(\opn{Mod} A)$:
\[
\mrm{R}\Gamma_{\mfrak{a}} M \cong \opn{Tel}(A;\mathbf{a}) \otimes_A M
\]
and
\[
\mrm{L}\Lambda_{\mfrak{a}} M \cong \opn{Hom}_A(\opn{Tel}(A;\mathbf{a}),M).
\]
\item \textbf{MGM equivalence}: For any $M\in \mrm{D}(\opn{Mod} A)$ the functorial maps
\[
\mrm{R}\Gamma_{\a} (\tau^L_M) : \mrm{R}\Gamma_{\a}(M) \to \mrm{R}\Gamma_{\a}(\mrm{L}\Lambda_{\a} (M) )
\]
and
\[
\mrm{L}\Lambda_{\a} (\sigma^R_M) : \mrm{L}\Lambda_{\a} (\mrm{R}\Gamma_{\a}(M)) \to \mrm{L}\Lambda_{\a} (M)
\]
are isomorphisms.
\item \textbf{Greenlees-May duality}: For any $M,N\in \mrm{D}(\opn{Mod} A)$ the morphisms
\[ \begin{aligned}
& \mrm{R}\opn{Hom}_A \bigl( \mrm{R} \Gamma_{\a} (M), \mrm{R} \Gamma_{\a} (N) \bigr)
\xar{\mrm{R}\opn{Hom}(1, \sigma^{\mrm{R}}_N)}
\mrm{R}\opn{Hom}_A \bigl( \mrm{R} \Gamma_{\a} (M), N \bigr) 
\\
& \qquad\xar{\mrm{R}\opn{Hom}(1, \tau^{\mrm{L}}_N)}
\mrm{R}\opn{Hom}_A \bigl( \mrm{R} \Gamma_{\a} (M), \mrm{L} \Lambda_{\a} (N) \bigr) 
\xleftarrow{\mrm{R}\opn{Hom}(\sigma^{\mrm{R}}_M, 1)}
\\ & \qquad  \qquad
\mrm{R}\opn{Hom}_A \bigl( M, \mrm{L} \Lambda_{\a} (N) \bigr)
\xleftarrow{\mrm{R}\opn{Hom}(\tau^{\mrm{L}}_M, 1)}
\mrm{R}\opn{Hom}_A \bigl( \mrm{L} \Lambda_{\a} (M), \mrm{L} \Lambda_{\a} (N) \bigr)
\end{aligned} \]
in $\mrm{D}(\opn{Mod} A)$ are isomorphisms.
\end{enumerate}
\end{thm}
\begin{proof}
\begin{enumerate}
\item The first claim is \cite{PSY1}, Proposition 4.8. The second claim is \cite{PSY1}, Corollary 4.25.
\item This is \cite{PSY1}, Lemma 6.2 and Lemma 6.6.
\item This is \cite{PSY1}, Theorem 6.12.
\end{enumerate}
\end{proof}

\begin{rem}
Under additional mild assumptions, this theorem was proved in \cite{AJL}. See also \cite{SC}. Replacing total derived functors by classical derived functor, this result originated in \cite{GM2}, generalizing the classical Matlis duality.
\end{rem}

\begin{rem}
By \cite{SC}, Theorem 1.1, the first statement in the above theorem is actually equivalent to $\a$ being weakly proregular.
\end{rem}

\begin{dfn}
A preadic ring $(A,\a)$ is called a weakly proregular preadic ring if $\a$ is weakly proregular. The category of weakly proregular preadic rings, denoted by $\mathcal{PAW}$ is a full subcategory of the category $\mathcal{PA}$ of preadic rings. The full subcategory of $\mathcal{PAW}$ which contains the weakly proregular preadic rings $(A,\a)$ such that $\Lambda_{\a}(A)$ is noetherian is denoted by $\mathcal{PAWN}$.
\end{dfn}

\begin{exa}
Let $\k$ be a noetherian ring. Suppose that $(A,\a)$ and $(B,\b)$ are flat preadic $\k$-algebras which are essentially formally of finite type. Then $(A\otimes_{\k} B, \a\otimes_{\k} B + A\otimes_{\k} \b)$ is a preadic ring with noetherian completion. We will show below that $(A\otimes_{\k} B, \a\otimes_{\k} B + A\otimes_{\k} \b) \in \mathcal{PAWN}$.
\end{exa}

\begin{rem}\label{remark-wpr-flat}
If $(A,\a)$ is a weakly proregular preadic ring, $(B,\b)$ is a preadic flat $(A,\a)$-algebra, and $\b = \a \cdot B$, then $\b$ is also weakly proregular. This is because any injective $B$-module is also injective over $A$, and since if $\mathbf{a}$ is a finite sequence that generates $\a$, and $\mathbf{b}$ is its image in $B$, then there is an isomorphism of complexes $\opn{Tel}(A;\mathbf{a}) \otimes_A B \cong \opn{Tel}(B;\mathbf{b})$. This fact was first observed in \cite{AJL}, Example 3.0(B).
\end{rem}

Recall that if $(A,\a)$ is a preadic noetherian ring, and if $I$ is an injective $A$-module, then $\Gamma_{\a}(I)$ is also an injective $A$-module. (\cite{HA}, Lemma 3.2). We now state and prove a weaker form of this fact in the case when $A$ is not necessarily noetherian, but $\Lambda_{\a}(A)$ is.

If $(A,\a)$ is a preadic ring, and $M$ is an $A$-module, then $M$ is called $\a$-flasque if for each $k>0$, we have that $\mrm{H}_{\a}^k(M) = 0 $, where $\mrm{H}_{\a}^k(M):= \mrm{H}^k(\mrm{R}\Gamma_{\a}(M))$. Any injective module is $\a$-flasque. If $M$ is $\a$-flasque, then the canonical morphism $\Gamma_{\a} (M) \to \mrm{R}\Gamma_{\a} (M)$ is an isomorphism. (\cite{YZ2}, Proposition 1.21). The direct limit of $\a$-flasque modules is $\a$-flasque. (\cite{YZ2}, Proposition 1.20).

\begin{prop}\label{prop-tor-of-inj-is-flasque}
Let $(A,\a)$ be a preadic ring. Let $\b\subseteq A$ be a finitely generated ideal. Suppose that the ring $\widehat{A} = \Lambda_{\mfrak{a}}(A)$ is noetherian. Let $\widehat{\mfrak{b}} = \mfrak{b}\widehat{A}$. Then for any injective $A$-module $I$, the $\widehat{A}$-module $\Gamma_{\mfrak{a}} I$ is $\widehat{\mfrak{b}}$-flasque.
\end{prop}
\begin{proof}
Let $A_j = A/\mfrak{a}^{j+1}$. Since $\mfrak{a}$ is finitely generated, there is an isomorphism $A_j \cong \widehat{A}/(\mfrak{a}\widehat{A})^{j+1}$. Note that by assumption, $A_j$ is noetherian. Let $\widehat{\mfrak{b}}_j$ be the image of $\widehat{\mfrak{b}}$ in $A_j$. Let $I_j = \opn{Hom}_A(A_j,I)$. Then $\Gamma_{\mfrak{a}} I = \varinjlim I_j$, so it is enough to show that $I_j$ is $\widehat{\mfrak{b}}$-flasque. Note also that $I_j$ is an injective $A_j$-module. Let $k>0$,  
let $\widehat{\mathbf{b}}$ be a finite sequence generating $\widehat{\mfrak{b}}$, and let $\mathbf{b}_j$ be its image in $A_j$.
Since $\widehat{A}$ is noetherian, $\widehat{\mfrak{b}}$ is weakly proregular, so that
\[
H^k_{\widehat{\mfrak{b}}} (I_j) \cong H^k(\opn{Tel}(\widehat{A};\widehat{\mathbf{b}}) \otimes_{\widehat{A}} I_j) \cong
H^k(\opn{Tel}(A_j;\mathbf{b}_j)\otimes_{A_j} I_j) \cong
H^k_{\widehat{\mfrak{b}}_j}(I_j)
\]
where the last isomorphism follows from the fact that $A_j$ is noetherian, so that $\widehat{\mfrak{b}}_j$ is weakly proregular. Since $I_j$ is injective over $A_j$, it follows that $H^k_{\widehat{\mfrak{b}}_j}(I_j) = 0$ for all $k>0$, which proves the claim. 
\end{proof}

Here is the main result of this section:

\begin{thm}\label{wpr-thm}
Let $(A,\a) \in \mathcal{PAWN}$ be a preadic weakly proregular ring with noetherian completion. Let $\b \subseteq A$ be a finitely generated ideal containing $\a$. Then $\b$ is weakly proregular, and $(A,\b) \in \mathcal{PAWN}$.
\end{thm}
\begin{proof}
We keep the notations of the above proposition. It is clear that $A/\b$ is noetherian, so it is enough to show that $\b$ is weakly proregular.
Let $\mathbf{a}$ be a finite sequence generating $\mfrak{a}$, and let $\mathbf{b}$ be a finite sequence generating $\mfrak{b}$. Let $I$ be an injective $A$-module. It is enough to show that $H^k(\opn{Tel}(A;\mathbf{b}) \otimes_A I) = 0$ for all $k > 0$. 
\par
Since $\mfrak{a}\subseteq\mfrak{b}$, the ideal generated by the concatenated sequence $(\mathbf{a},\mathbf{b})$ is equal to the ideal generated by $\mathbf{b}$. Hence, there is a homotopy equivalence
$\opn{Tel}(A;(\mathbf{a},\mathbf{b})) \cong \opn{Tel}(A;\mathbf{b})$.
Hence, there is an isomorphism in $\mrm{D}(\opn{Mod} A)$
\[
\opn{Tel}(A;\mathbf{b}) \otimes_A I \cong
\opn{Tel}(A;(\mathbf{a},\mathbf{b})) \otimes_A I \cong
\opn{Tel}(A;\mathbf{b}) \otimes_A \opn{Tel}(A;\mathbf{a}) \otimes_A I 
\]
Since $\mathbf{a}$ is a weakly proregular sequence, $I$ is an injective $A$-module, and $\opn{Tel}(A;\mathbf{b})$ is a bounded complex of flat modules, the latter is isomorphic in $\mrm{D}(\opn{Mod} A)$ to
\[
\opn{Tel}(A;\mathbf{b}) \otimes_A \Gamma_{\mfrak{a}} I
\]
Thus, it is enough to show that all the cohomologies (except the zeroth) of the complex of $A$-modules $\opn{Tel}(A;\mathbf{b}) \otimes_A \Gamma_{\mfrak{a}} I$ vanish. Note that since $\Gamma_{\mfrak{a}} I \in \opn{Mod}(\widehat{A})$, this complex also has the structure of a complex of $\widehat{A}$-modules. There is an isomorphism of complexes
\[
\opn{Tel}(A;\mathbf{b}) \otimes_A \Gamma_{\mfrak{a}} I \cong
\opn{Tel}(A;\mathbf{b}) \otimes_A (\widehat{A} \otimes_{\widehat{A}} \Gamma_{\mfrak{a}} I) \cong (\opn{Tel}(A;\mathbf{b})\otimes_A \widehat{A}) \otimes_{\widehat{A}}\Gamma_{\mfrak{a}} I
\]
Letting $\widehat{\mathbf{b}}$ be the image of the sequence $\mathbf{b}$ under the map $A\to \widehat{A}$, we obtain an isomorphism of complexes
\[
(\opn{Tel}(A;\mathbf{b})\otimes_A \widehat{A}) \otimes_{\widehat{A}}\Gamma_{\mfrak{a}} I \cong
\opn{Tel}(\widehat{A};\widehat{\mathbf{b}}) \otimes_{\widehat{A}} \Gamma_{\mfrak{a}}I
\]
Consider the image of $\opn{Tel}(\widehat{A};\widehat{\mathbf{b}}) \otimes_{\widehat{A}} \Gamma_{\mfrak{a}}I$ in $\mrm{D}(\opn{Mod}\widehat{A})$. By weak proregularity of the sequence $\widehat{\mathbf{b}}$, there is an isomorphism in $\mrm{D}(\opn{Mod}\widehat{A})$:
\[
\opn{Tel}(\widehat{A};\widehat{\mathbf{b}}) \otimes_{\widehat{A}} \Gamma_{\mfrak{a}}I \cong
\mrm{R}\Gamma_{\widehat{\mfrak{b}}} \Gamma_{\mfrak{a}}I
\]
By the above proposition, this is isomorphic in $\mrm{D}(\opn{Mod}\widehat{A})$ to
\[
\Gamma_{\widehat{\mfrak{b}}} \Gamma_{\mfrak{a}}I
\]
Since this complex is clearly concentrated in degree zero, it follows that all of its cohomologies except the zeroth vanish, which proves the result.
\end{proof}

\begin{cor}\label{cor-wpr-of-efft}
Let $\k$ be a noetherian ring. Let $(A,\a)$ be a flat preadic $\k$-algebra which is essentially formally of finite type. Let $(B,\b)$ be a flat preadic $\k$-algebra which is noetherian. Then $(A\otimes_{\k} B, \a \otimes_{\k} B + A\otimes_{\k} \b) \in \mathcal{PAWN}$.
\end{cor}
\begin{proof}
Since $\k$ is noetherian, and $A$ is essentially formally of finite type over $\k$, $A$ is also noetherian. Since $B$ is noetherian, $\b$ is weakly proregular. Since $B$ is flat over $\k$, the map $B \to A\otimes_{\k} B$ is also flat. It follows from Remark \ref{remark-wpr-flat}, that the ideal $A\otimes_{\k} \b \subseteq A\otimes_{\k} B$ is weakly proregular. Because of flatness of $\k \to A$, it follows that $A\otimes_{\k} \b = \ker (A\otimes_{\k} B \to A\otimes_{\k} B/\b)$. Since $B/\b$ is essentially of finite type over $\k$, it follows that $A\otimes_{\k} B/\b$ is noetherian. Hence, $\Lambda_{A\otimes_{\k} \b} (A\otimes_{\k} B)$ is also noetherian. The claim now follows from the above theorem.
\end{proof}

We end this section with a dual result of Proposition \ref{prop-tor-of-inj-is-flasque}, replacing injective modules with projective modules, and flasque modules with flat modules:
\begin{prop}\label{prop-completion-of-projective-is-flat}
Let $(A,\a)$ be a preadic ring. Suppose that the ring $\Lambda_{\a}(A)$ is noetherian. Let $P$ be a projective $A$-module. Then $\Lambda_{\a}(P)$ is a flat $\Lambda_{\a}(A)$-module.
\end{prop}
\begin{proof}
Suppose $P\oplus K \cong F$, where $F$ is a free $A$-module. Since the completion functor is additive, and since a direct summand of a flat-module is flat, it is enough to show the claim for $F$. This follows immediately from combining \cite{YE1}, Theorem 2.7, and \cite{YE1}, Theorem 3.4(2).
\end{proof}
\section{Completion and torsion over Differential Graded Algebras}

The purpose of this section is to extend the Greenlees-May duality and the Matlis-Greenlees-May equivalence (Theorem \ref{thm-ring-mgm}) to the setting of super-commutative differential graded algebras. 
We begin by recalling some basic definitions and results concerning super-commutative differential graded algebras. In general, we shall follow the conventions of \cite{YZ1}, regarding DG-algebras. A more detailed account using the same notations can be found in \cite{GM1},  Chapter V.3, in a more restricted situation when the DG-algebras are over a field of characteristic  $0$. See also \cite{ML}, Chapter VI.7, where the theory is introduced for positive DG-algebras.

A ($\mathbb{Z}$-)graded algebra $A = \oplus_{i\in \mathbb{Z}} A^i$ is said to be super-commutative if for all $a \in A^i$ and $b \in A^j$, we have that $a\cdot b = (-1)^{ij} b\cdot a$, and $a^2 = 0$ if $i$ is odd. 
A graded algebra $A = \oplus_{i\in \mathbb{Z}} A^i$ is said to be non-positive if $A^i = 0$ for all $i>0$. A super-commutative non-positive graded algebra $A = \oplus_{i\le 0} A^i$ together with a derivation $d:A\to A$ of degree $+1$, such that $d\circ d = 0$, and such that $d(ab) = d(a)b + (-1)^iad(b)$ for any $a \in A^i$ and $b \in A^j$, is called a differential graded algebra or a DG-algebra. 
All DG-algebras in this paper are assumed to be super-commutative and non-positive.
If $A$ is a DG-algebra then the cohomology of $A$, $HA = \oplus_{i\le 0} H^i A$ is a graded-algebra. A DG-algebra homomorphism $u:A\to B$ is a homomorphism of graded algebras which commutes with the differentials. 
If $H(u):H(A)\to H(B)$ is an isomorphism of graded algebras then $u$ is said to be a quasi-isomorphism. 
Given a DG-algebra $A$, a DG $A$-module $M$ is a graded $A$-module $M = \oplus_{i \in \mathbb{Z}} M^i$, endowed with a degree $1$ $\mathbb{Z}$-linear homomorphism $d:M\to M$ such that $d(am) = d(a)m +(-1)^ia d(m)$ for $a \in A^i$ and $m\in M^j$, and such that $d \circ d =0$. 
A morphism $f:M\to N$ between two DG $A$-modules is an homogeneous morphism of degree $0$ of graded modules which commutes with the differentials. The category of DG $A$-modules is an abelian category, denoted by $\opn{DGMod} A$. A morphism $f:M \to N$ between two DG $A$-modules is called a quasi-isomorphism if it induces an isomorphism of graded-modules on cohomology. 
If $A$ is a DG-algebra, then $A^0$ is a commutative ring. We will consider rings as DG-algebras concentrated in degree zero. If $A$ is a ring, a DG $A$-module is a complex over $A$. For any DG-algebra $A$, any DG $A$-module is a complex over the ring $A^0$. In particular, $A$ is a (bounded above) complex over $A^0$.
Given two DG $A$-modules $M,N$, one can construct the DG $A$-modules $\opn{Hom}_A(M,N)$ and $M\otimes_A N$. See \cite{ML}, Chapter VI.7 for details of these constructions. The usual hom-tesnor adjunction is satisfied in this setting. As in the case of rings, we say that a DG $A$-module $P$ (respectively I) is K-projective (resp. K-injective) if for any acyclic DG $A$-module $M$, the DG $A$-module $\opn{Hom}_A(P,M)$ (resp. $\opn{Hom}_A(M,I)$) is acyclic. A DG $A$-module $F$ is called K-flat if for any acyclic DG $A$-module $M$, the DG $A$-module $F\otimes_A M$ is acyclic.
\par
Given a DG-algebra $A$, we denote by $\widetilde{\mrm{D}}(\opn{DGMod} A)$ the derived category obtained by inverting all its quasi-isomorphisms. See \cite{KE} for a detailed account about this construction. If $A$ is a ring, then $\mrm{D}(\opn{Mod} A) = \widetilde{\mrm{D}}(\opn{DGMod} A)$. For any DG-algebra $A$, the DG-algebra map $A^0 \to A$ give rise to forgetful functors 
$Q_A:\opn{DGMod} A \to \mrm{C}(\opn{Mod} A^0)$ and
$Q_A:\widetilde{\mrm{D}}(\opn{DGMod} A) \to \mrm{D}(\opn{Mod} A^0)$.

\begin{rem}\label{remark-lifting-quasi}
Let $A$ be a DG-algebra, and let $f:M\to N$ be a morphism between two DG $A$-modules. Then it is clear that $f$ is a quasi-isomorphism if and only if $Q_A(f):Q_A(M)\to Q_A(N)$ is a quasi-isomorphism.
\end{rem}

We now discuss the adic completion and torsion functors on the category of differential graded modules. 

\begin{dfn}
Let $A$ be a DG-algebra. Let $\mf{a} \subseteq A^0$ be a finitely generated ideal. Define $\Lambda_{\mf{a}},\Gamma_{\mf{a}}:\opn{DGMod}(A) \to \opn{DGMod}(A)$ by
\[
\Lambda_{\mf{a}} M = \varprojlim (M\otimes_{A^0} {A^0}/\mf{a}^n)
\]
and
\[
\Gamma_{\mf{a}} M = \varinjlim \opn{Hom}_{A^0}(A^0/\mf{a}^n,M)
\]
\end{dfn}

The limits in the above definition are taken in the category $\opn{DGMod}(A)$, so that $(\Lambda_{\mfrak{a}} M)^i = \Lambda_{\mfrak{a}}(M^i)$ and 
$(\Gamma_{\mfrak{a}} M)^i = \Gamma_{\mfrak{a}}(M^i)$ for each $i$.

\begin{rem}
It is easy to verify that over a super-commutative non-positive DG-algebra, $\Lambda_{\mf{a}} M$ and $\Gamma_{\mf{a}} M$ have the structure of DG $A$-modules.
\end{rem}

\begin{rem}
It is clear from the above definition that the forgetful functor $Q_A$ commutes with the $\Lambda_{\mf{a}}$ and $\Gamma_{\mf{a}}$ functors.
\end{rem}

\begin{rem}
We focus on completion and torsion with respect to ideals generated by degree $0$ elements. More generally, one might consider completion and torsion with respect to more complicated DG-ideals. 
\end{rem}

\begin{rem}
The DG-module $\Lambda_{\mf{a}} (A) = \oplus_{i\le 0} \Lambda_{\mf{a}} (A^i)$ has a structure of a (super-commutative non-positive) DG-algebra. Furthermore, for any DG $A$-module $M$, the DG $A$-modules $\Gamma_{\mf{a}} M$ and $\Lambda_{\mf{a}} M$ have the structure of DG $\Lambda_{\mf{a}}(A)$-modules. Thus, we also have functors $\Lambda_{\mf{a}},\Gamma_{\mf{a}}:\opn{DGMod}(A) \to \opn{DGMod}(\Lambda_{\mf{a}}(A))$.
\end{rem}
Similarly to Definition \ref{preadicdef}, given a DG-algebra $A$ and a finitely generated ideal $\a \subseteq A^0$, we will call the pair $(A,\a)$ a preadic DG-algebra. A map $f:(A,\a) \to (B,\b)$ will be called preadic if the map $f^0:(A^0,\a) \to (B^0,\b)$ is preadic. A preadic DG-algebra $(A,\a)$ such that the canonical map $\tau_A : A\to \Lambda_{\a}(A)$ is bijective is called an adic DG-algebra. A preadic DG-algebra $(A,\a)$ is called a weakly proregular preadic DG-algebra if the ideal $\a \subseteq A^0$ is weakly proregular.
We now wish to extend the homological theory of completion and torsion to the above setting. First, we restrict our attention to a class of DG-algebras that is wide enough for our applications, which is easier to work with from an homological point of view. 

As far as we know, the next definition is new:
\begin{dfn}\label{dfn-internally-flat}
\begin{enumerate}
\item A DG-algebra $A$ is said to be \textbf{internally projective} if it is a K-projective complex over $A^0$.
\item A DG-algebra $A$ is said to be \textbf{internally flat} if it is a K-flat complex over $A^0$.
\end{enumerate}

\end{dfn}
Clearly, any internally projective DG-algebra is internally flat. Any ring is internally projective. If for each $i<0$ the $A^0$-module $A^i$ is projective (respectively flat), then $A$ is internally projective (resp. internally flat).

These special types of DG-algebras are interesting from an homological point of view due to the next easy proposition:
\begin{prop}\label{prop-internal-flat-preserve-homological}
Let $A$ be a DG-algebra, and let $F,P$ and $I$ be K-flat, K-projective and K-injective DG-modules over $A$. If $A$ is internally flat then $Q_A(F)$ is a K-flat complex over $A^0$ and $Q_A(I)$ is a K-injective complex over $A^0$. If $A$ is internally projective then $Q_A(P)$ is a K-projective complex over $A^0$.
\end{prop}
\begin{proof}
Let $M$ be an acyclic complex over $A^0$. There is a sequence of isomorphisms of complexes:
\[
\opn{Hom}_{A^0}(M,Q_A(I)) \cong \opn{Hom}_{A^0}(M,\opn{Hom}_A(A,I)) \cong \opn{Hom}_A(M\otimes_{A^0} A, I)
\]
Since $A$ is K-flat over $A^0$, is follows that the DG $A$-module $M\otimes_{A^0} A$ is acyclic. Since $I$ is K-injective over $A$, it now follows that $I$ is K-injective over $A^0$.
The other claims are proved similarly, using the associativity of the tensor product.
\end{proof}

We now discuss preadic DG-algebra resolutions, and show that there are enough internally projective DG-algebras from the point of view of the category of adic noetherian rings, and essentially formally of finite type preadic maps.

\begin{dfn}
Let $(A,\a)$ and $(B,\b)$ be two preadic DG-algebras. A preadic map $f:(A,\a) \to (B,\b)$ is called an \textbf{adic quasi-isomorphism} if it is a quasi-isomorphism, and $\a \cdot B^0$ is an ideal of definition of the preadic ring $(B^0,\b)$. We say that $(A,\a)$ is an adic DG-algebra resolution of $(B,\b)$.
\end{dfn}

\begin{prop}\label{prop-existence-of-res}
\begin{enumerate}
\item Let $(B,\b)$ be a preadic DG-algebra. Then there is an internally projective preadic DG-algebra $(A,\a)$, and an adic quasi-isomorphism $(A,\a) \to (B,\b)$. If moreover $\b$ is weakly proregular, then $\a$ can also be chosen to be weakly proregular.
\item Let $\k$ be a noetherian ring, and let $(A,\a)$ be an adic ring which is an essentially formally of finite type $\k$-algebra. Then there is an internally projective weakly proregular preadic DG-algebra $(\widetilde{A},\widetilde{\a})$, such that the following holds: the map $\k \to A$ factors as $\k \to \widetilde{A} \to A$, where $\k \to \widetilde{A}$ is K-flat, $\k \to (\widetilde{A}^0,\widetilde{\a})$ is essentially formally of finite type, $(\widetilde{A},\widetilde{\a}) \to (A,\a)$ is an adic quasi-isomorphism, the ring $\widetilde{A}^0$ is noetherian, and for each $i<0$, $\widetilde{A}^i$ is a finitely generated projective $\widetilde{A}^0$-module.
\end{enumerate}
\end{prop}
\begin{proof}
\begin{enumerate}
\item Applying the construction of a semi-free resolution of the DG-algebra map $B^0\to B$ in \cite{YZ1}, Proposition 1.7(1), we get a DG-algebra $A$, and a quasi-isomorphism $f:A\to B$. Moreover, by the construction in this particular case, $A^0 = B^0$, $f^0 = 1_{B^0}$, and for each $i<0$, $A^i$ is a free $A^0$-module. Taking $\a = \b$, we get the required adic DG-algebra resolution, and moreover, if $\b$ is weakly proregular, then so is $\a$.
\item By \cite{LNS}, Lemma 2.4.3, there is a factorization $\k \to \widetilde{A}^0 \to A$, of $\k \to A$, where $(\widetilde{A}^0,\widetilde{\a})$ is a noetherian adic ring, $\k \to \widetilde{A}^0$ is flat and essentially formally of finite type, $\a = A\cdot \widetilde{\a}$, and the map $\widetilde{A}^0 \to A$ is surjective. Applying the construction of a semi-free resolution of the finite ring map $\widetilde{A}^0 \to A$ in \cite{YZ1}, Proposition 1.7(3), we get a resolution $\k \to \widetilde{A} \to A$ with all the required properties.
\end{enumerate}
\end{proof}

Let $(A,\a)$ be a preadic DG-algebra. As for any additive functors, the derived functors $\mrm{R}\Gamma_{\a}, \mrm{L}\Lambda_{\a} :\widetilde{\mrm{D}}(\opn{DGMod} A) \to \widetilde{\mrm{D}}(\opn{DGMod} A)$ of the functors $\Gamma_{\a}$ and $\Lambda_{\a}$ exist. They are calculated using K-injective and K-projective resolutions respectively.

\begin{prop}
Let $(A,\a)$ be an internally flat preadic DG-algebra. Then the derived functors $\mrm{R}\Gamma_{\a}, \mrm{L}\Lambda_{\a}$ commute with the forgetful functor $Q_A$.
\end{prop}
\begin{proof}
Let $M$ be a DG $A$-module. Let $P \cong M$ be a K-projective resolution of $M$ 
over $A$. Because of internal flatness, $Q_A(P)$ is a K-flat complex over $A^0$. (Note that we do not necessarily know that $Q_A(P)$ is K-projective over $A^0$). Hence, there is a sequence of functorial isomorphisms:
\[
Q_A(\mrm{L}\Lambda_{\a} M) \cong Q_A(\Lambda_{\a} P) = \Lambda_{\a}(Q_A(P)) \cong \mrm{L}\Lambda_{\a}(Q_A(P)) \cong \mrm{L}\Lambda_{\a} (Q_A(M))
\]
The proof of the claim for $\mrm{R}\Gamma_{\a}$ is similar.
\end{proof}

\begin{prop}
Let $(A,\a)$ be an internally flat preadic DG-algebra. Let $P$ be an acyclic K-flat DG $A$-module. Then the DG $A$-module $\Lambda_{\mf{a}} P$ is also acyclic. Thus, we can calculate $\mrm{L}\Lambda_{a}$ over $A$ using K-flat resolutions.
\end{prop}
\begin{proof}
Since $A$ is internally flat, $Q_A(P)$ is a K-flat acyclic complex over $A^0$.
By \cite{AJL}, the complex $\Lambda_{\a}(Q_A(P))$ is acyclic, hence the complex $Q_A(\Lambda_{\a}(P))$ is acyclic. Since taking cohomology commutes with $Q_A$, the complex $\Lambda_{\a}(P)$ is also acyclic. The last claim now follows from \cite{RD}, Theorem I.5.1.
\end{proof}

\begin{rem}
As in the case of rings, given a preadic DG-algebra $(A,\a)$, there are canonical maps
\[
\tau^L_M: M \to \mrm{L}\Lambda_{\mf{a}}(M)
\]
and 
\[
\sigma^R_M: \mrm{R}\Gamma_{\mfrak{a}}(M) \to M. 
\]
The construction of these maps over $A^0$ was done in \cite{PSY1}, Proposition 2.7 and Proposition 2.10. This construction is identical if $A$ is a DG-algebra. If furthermore, $A$ is internally flat, then the image of $\tau^L_M$ and $\sigma^R_M$ under the functor $Q_A$ is equal to the maps described in \cite{PSY1}.
\end{rem}

We now turn to generalize Theorem \ref{thm-ring-mgm} to the setting of internally flat DG-algebras. We start with the first statement of claim (1) of the theorem.

If $A^0$ is a ring, and $\mathbf{a} = (a_1,\dots,a_n)$ is a finite sequence of elements in $A^0$, recall that the infinite dual Koszul complex (\cite{PSY1}, Section 3) is defined as follows: First, for $n=1$, one sets
\[
\opn{K}^{\vee}_{\infty}(A^0; (a)) := 
\bigl( \cdots \to 0 \to A^0 \xar{d} A^0[{a}^{-1}] \to 0 \to
\cdots \bigr) 
\]
concentrated in degrees $0$ and $1$. Here, the map $d$ is the localization map. In general, one sets:
\[
\opn{K}^{\vee}_{\infty}(A^0; \mathbf{a}) :=
\opn{K}^{\vee}_{\infty}(A^0; (a_1)) \otimes_{A^0} \opn{K}^{\vee}_{\infty}(A^0; (a_2)) \otimes_{A^0} \dots \otimes_{A^0} \opn{K}^{\vee}_{\infty}(A^0; (a_n)).
\]

If $f:A^0\to B^0$ is a ring map, and if $\mathbf{b} = f(\mathbf{a})$, then there is an isomorphism of complexes
\[
\opn{K}^{\vee}_{\infty}(A^0; \mathbf{a}) \otimes_{A^0} B^0 \cong \opn{K}^{\vee}_{\infty}(B^0; \mathbf{b}).
\]

Suppose $\a$ is the ideal generated by $\mathbf{a}$. In \cite{PSY1}, equation (3.19), a functorial morphism of complexes $v_{\mathbf{a},M} : \Gamma_{\a} (M) \to \opn{K}^{\vee}_{\infty}(A^0; \mathbf{a}) \otimes_{A^0} M$ was constructed. Let us recall the construction:

Again, for $n=1$, if $M$ is a single $A^0$-module, the inclusion map $\Gamma_{(a)}(M) \inj M$ give rise to a morphism of complexes
\[\xymatrixcolsep{4pc}
\xymatrix{
0 \ar[r] & \Gamma_{(a)} (M) \ar[r] \ar[d] & 0\\
0 \ar[r] &      M \ar[r] & M[a^{-1}] \ar[r] & 0
}
\]
which we denote by $v_{(a),M} : \Gamma_{(a)} (M) \to \opn{K}^{\vee}_{\infty}(A^0; (a)) \otimes_{A^0} M$. By totalization, we get an $A^0$-linear map $v_{(a),M} : \Gamma_{(a)} (M) \to \opn{K}^{\vee}_{\infty}(A^0; (a)) \otimes_{A^0} M$ for any complex $M \in \mrm{C}(\opn{Mod} A^0)$.

Composing the map $v_{(a_2),\Gamma_{(a_1)}(M)}: \Gamma_{(a_2)} ( \Gamma_{(a_1)} (M)) \to \opn{K}^{\vee}_{\infty}(A^0; (a_2)) \otimes_{A^0} \Gamma_{(a_1)}(M)$ with the map $1\otimes_{A^0} v_{(a_1),M}$ we obtain a map $v_{(a_1,a_2)} : \Gamma_{(a_1,a_2)} (M) \to \opn{K}^{\vee}_{\infty}(A^0; (a_1,a_2)) \otimes_{A^0} \Gamma_{(a_1)}(M)$.
Proceeding in this way, one gets the map of complexes $v_{\mathbf{a},M} : \Gamma_{\a} (M) \to \opn{K}^{\vee}_{\infty}(A^0; \mathbf{a}) \otimes_{A^0} M$.

Note further that $(\opn{K}^{\vee}_{\infty}(A^0; \mathbf{a}))^0 = A^0$, so the identity map $1_{A^0}$ gives a map of complexes $e_{\mathbf{a},M}: \opn{K}^{\vee}_{\infty}(A^0; \mathbf{a}) \otimes_{A^0} M \to M$. Clearly, for any complex $M$, there is an equality $e_{\mathbf{a},M} \circ v_{\mathbf{a},M} = \sigma_M: \Gamma_{\a}(M) \to M$.

Assume now that $(A,\a)$ is a preadic DG-algebra, and suppose that $\mathbf{a} = (a_1,\dots,a_n)$ is a finite sequence that generates $\a$. Let $M$ be a DG $A$-module. The inclusion map $\Gamma_{\a} (M) \to M$ is clearly an $A$-linear map. Hence, it follows that the above construction is also $A$-linear, so that the two map $v_{\mathbf{a},M} : \Gamma_{\a} (M) \to \opn{K}^{\vee}_{\infty}(A^0; \mathbf{a}) \otimes_{A^0} M$ and $e_{\mathbf{a},M} :\opn{K}^{\vee}_{\infty}(A^0; \mathbf{a}) \otimes_{A^0} M \to M$ are morphisms of DG $A$-modules.

\begin{lem}\label{lem-gamma-of-inj}
Let $(A,\a)$ be an internally flat weakly proregular preadic DG-algebra. Let $I$ be a K-injective DG $A$-module. Let $\mathbf{a}$ be a finite sequence of elements of $A^0$ that generates $\a$. Then the $A$-linear morphism
\[
v_{\mathbf{a},I} : \Gamma_{\a} (I) \to \opn{K}^{\vee}_{\infty}(A^0; \mathbf{a}) \otimes_{A^0} I
\]
is a quasi-isomorphism.
\end{lem}
\begin{proof}
By Remark \ref{remark-lifting-quasi}, it is enough to show that the map $Q_A(v_{\mathbf{a},I}) : Q_A(\Gamma_{\a} (I)) \to Q_A(\opn{K}^{\vee}_{\infty}(A^0; \mathbf{a}) \otimes_{A^0} I)$ is an $A^0$-linear quasi-isomorphism. 
Clearly, $Q_A(\Gamma_{\a} (I)) = \Gamma_{\a}(Q_A(I))$,  $Q_A(\opn{K}^{\vee}_{\infty}(A^0; \mathbf{a}) \otimes_{A^0} I) =
\opn{K}^{\vee}_{\infty}(A^0; \mathbf{a}) \otimes_{A^0} Q_A(I)$, and 
$Q_A(v_{\mathbf{a},I}) = v_{\mathbf{a},Q_A(I)}$. Because of internal flatness, $Q_A(I)$ is K-injective over $A^0$. The result now follows from \cite{PSY1}, Corollary 3.25.
\end{proof}

\begin{thm}\label{thm-rgamma}
Let $(A,\a)$ be an internally flat weakly proregular preadic DG-algebra. Let $\mathbf{a}$ be a finite sequence of elements of $A^0$ that generates $\a$.
Then there is a functorial isomorphism
\[
v^{\mrm{R}}_{\mathbf{a},M} :\mrm{R}\Gamma_{\mfrak{a}} M \to \opn{K}^{\vee}_{\infty}(A^0; \mathbf{a}) \otimes_{A^0} M
\]
for any $M\in \widetilde{\mrm{D}}(\opn{DGMod} A),$
making the diagram in $\widetilde{\mrm{D}}(\opn{DGMod} A)$
\[
\xymatrix{
\mrm{R}\Gamma_{\mfrak{a}} M \ar[r]\ar[rd]_{\sigma_M^R} & \opn{K}^{\vee}_{\infty}(A^0; \mathbf{a}) \otimes_{A^0} M\ar[d]^{e_{\mathbf{a},M}}\\
& M
}
\]
commutative.
\end{thm}
\begin{proof}
Let $M\cong I$ be a K-injective resolution over $A$. The first claim now follows from the lemma. The equality $e_{\mathbf{a},I} \circ v_{\mathbf{a},I} = \sigma_I$, and the fact that $\sigma_I = \sigma^R_I$ since $I$ is K-injective, provides the second claim. 
\end{proof}

\begin{cor}\label{cor-rgamma-id}
Let $(A,\a)$ be an internally flat weakly proregular preadic DG-algebra. Then for any $M\in \widetilde{\mrm{D}}(\opn{DGMod} A),$ the morphism
\[
\sigma^R_{\mrm{R}\Gamma_{\a}(M)}:\mrm{R}\Gamma_{\a}(\mrm{R}\Gamma_{\a}(M)) \to \mrm{R}\Gamma_{\a}(M)
\]
is an isomorphism.
\end{cor}
\begin{proof}
This is because, according to \cite{PSY1}, Lemma 3.29, if $\mathbf{a}$ is a finite sequence in $A^0$ that generates $\a$, then the map
\[
\opn{K}^{\vee}_{\infty}(A^0; \mathbf{a}) \otimes_{A^0} \opn{K}^{\vee}_{\infty}(A^0; \mathbf{a}) \to \opn{K}^{\vee}_{\infty}(A^0; \mathbf{a})
\]
is a quasi-isomorphism.
\end{proof}

Next, we generalize the second claim of Theorem \ref{thm-ring-mgm}(1) to the setting of DG-algebras. 

Let $(A^0,\a)$ be a preadic ring, and let $\mathbf{a}$ be a finite sequence that generates $\a$. Recall that the telescope complex $\opn{Tel}(A^0;\mathbf{a})$ was defined in Section 1. Following \cite{PSY1}, Section 4, we recall there is a sequence of perfect complexes $\opn{Tel}_j(A^0;\mathbf{a})$ approximating it:
For each $j$, if $n=1$, $\opn{Tel}_j(A^0;(a))$ is the subcomplex
\[ \opn{Tel}_j(A^0; a) :=
\bigl( \cdots \to 0 \to \opn{F}([0, j] , A^0) \xar{\d}
\opn{F}([0, j] , A^0) \to 0 \to \cdots \bigr) \]
of $\opn{Tel}(A^0; a)$, where $\opn{F}([0, j] , A^0)$ is the free $A^0$-module with basis $\delta_0,\delta_1,\dots,\delta_j$. If $\mathbf{a} = (a_1,\dots,a_n)$, then 
\[ \opn{Tel}_j(A^0; \bsym{a}) :=
\opn{Tel}_j(A^0; a_1) \otimes_{A^0} \cdots \otimes_{A^0} \opn{Tel}_j(A^0; a_n). \]
It is clear that $\opn{Tel}(A^0;\mathbf{a}) = \varinjlim \opn{Tel}_j(A^0;\mathbf{a})$

In \cite{PSY1}, equation (4.12), an inverse system of homomorphisms of complexes $\opn{tel}_{\mathbf{a},j}: \opn{Hom}_{A^0}(\opn{Tel}_j(A^0;\mathbf{a}),A^0) \to A^0/{\a}^j$ was defined.

Assume now that $(A,\a)$ is a preadic DG-algebra. Let $M$ be a DG $A$-module. For each $j$, since $\opn{Tel}_j(A^0;\mathbf{a})$ is a perfect complex, there is an isomorphism of complexes $\opn{Hom}_{A^0}(\opn{Tel}_j(A^0;\mathbf{a}),M) \to \opn{Hom}_{A^0}(\opn{Tel}_j(A^0;\mathbf{a}),A^0) \otimes_{A^0} M$, functorial in $M$. Clearly, this isomorphism is $A$-linear. Composing this isomorphism with $\opn{tel}_{\mathbf{a},j}\otimes_{A^0} 1_M$, we see that there is an $A$-linear functorial map $\opn{tel}_{\mathbf{a},M,j}: \opn{Hom}_{A^0}(\opn{Tel}_j(A^0;\mathbf{a}),M) \to M\otimes_{A^0} A/{\a}^j$. Let $\opn{tel}_{\mathbf{a},M} := \varprojlim \opn{tel}_{\mathbf{a},M,j}: \opn{Hom}_{A^0}(\opn{Tel}(A^0;\mathbf{a}),M) \to \Lambda_{\a}(M)$. This map was constructed in \cite{PSY1}, Definition 4.16. The discussion here shows that this map is $A$-linear.

According to \cite{PSY1}, Lemma 4.7, there is a quasi-isomorphism $w_{\mathbf{a}}: \opn{Tel}(A^0;\mathbf{a}) \to \opn{K}^{\vee}_{\infty}(A^0; \mathbf{a})$. Composing it with the map $e_{\mathbf{a}}$ defined above, we obtain a map $u_{\mathbf{a}}: \opn{Tel}(A^0;\mathbf{a}) \to A^0$. Hence, for any DG $A$-module $M$, there is an $A$-linear map
\[
\opn{Hom}_{A^0}(u_{\mathbf{a}},1_M) : M \to \opn{Hom}_{A^0}(\opn{Tel}(A^0;\mathbf{a}),M).
\]
It is easy to verify that 
\[
\opn{tel}_{\mathbf{a},M} \circ \opn{Hom}_{A^0}(u_{\mathbf{a}},1_M) = \tau_M.
\]
\begin{lem}
Let $(A,\a)$ be an internally flat weakly proregular preadic DG-algebra. Let $P$ be a K-flat DG $A$-module. Let $\mathbf{a}$ be a finite sequence of elements of $A^0$ that generates $\a$. Then the $A$-linear morphism
\[
\opn{tel}_{\mathbf{a},P} :\opn{Hom}_{A^0}(\opn{Tel}(A^0;\mathbf{a}),P) \to \Lambda_{\a}(P)
\]
is a quasi-isomorphism.
\end{lem}
\begin{proof}
Identical to the proof of Lemma \ref{lem-gamma-of-inj}, by using \cite{PSY1}, Corollary 4.23.
\end{proof}

\begin{thm}\label{thm-llambda}
Let $(A,\a)$ be an internally flat weakly proregular preadic DG-algebra. Let $\mathbf{a}$ be a finite sequence of elements of $A^0$ that generates $\a$.
Then there is a functorial isomorphism
\[
\opn{tel}^{\mrm{L}}_{\mathbf{a},M}: \opn{Hom}_{A^0}(\opn{Tel}(A^0;\mathbf{a}),M) \to \mrm{L}\Lambda_{\mfrak{a}} M
\]
for any $M\in \widetilde{\mrm{D}}(\opn{DGMod} A),$
making the diagram in $\widetilde{\mrm{D}}(\opn{DGMod} A)$
\[
\xymatrix{
M \ar[d]_{\opn{Hom}_{A^0}(u_{\mathbf{a}},1_M)} \ar[rd]^{\tau^L_M} \\
\opn{Hom}_{A^0}(\opn{Tel}(A^0;\mathbf{a}),M) \ar[r] & \mrm{L}\Lambda_{\mfrak{a}} M
}
\]
commutative.
\end{thm}
\begin{proof}
Let $P\cong M$ be a K-flat resolution over $A$. Now continue as in the proof of Theorem \ref{thm-rgamma}, using the fact that $\tau_P = \tau^L_P$.
\end{proof}

\begin{cor}\label{cor-llambda-id}
Let $(A,\a)$ be an internally flat weakly proregular preadic DG-algebra. Then for any $M\in \widetilde{\mrm{D}}(\opn{DGMod} A),$ the morphism
\[
\tau^L_{\mrm{L}\Lambda_{\a}(M)}:\mrm{L}\Lambda_{\a}(M) \to \mrm{L}\Lambda_{\a}(\mrm{L}\Lambda_{\a}(M))
\]
is an isomorphism.
\end{cor}
\begin{proof}
Identical to the proof of Corollary \ref{cor-rgamma-id}, using \cite{PSY1}, Lemma 6.9.
\end{proof}

We now generalize Theorem \ref{thm-ring-mgm}(2) to the setting of DG-algebras:

\begin{thm}\label{thm-RgammaLLambda}
Let $(A,\a)$ be an internally flat weakly proregular preadic DG-algebra. Then for any $M\in \widetilde{\mrm{D}}(\opn{DGMod} A)$, the morphism
\[
\mrm{R}\Gamma_{\a}(\tau^L_M) : \mrm{R}\Gamma_{\a} (M) \to \mrm{R}\Gamma_{\a}(\mrm{L}\Lambda_{\a}(M))
\]
is an isomorphism in $\widetilde{\mrm{D}}(\opn{DGMod} A)$.
\end{thm}
\begin{proof}
Applying the $\mrm{R}\Gamma_{\a}$ functor to the diagram in Theorem \ref{thm-llambda}, we see that it is enough to show that the $A$-linear map
\[
\mrm{R}\Gamma_{\a}( \opn{Hom}_{A^0}(u_{\mathbf{a}},1_M) ): \mrm{R}\Gamma_{\a}(M) \to  \mrm{R}\Gamma_{\a}(\opn{Hom}_{A^0}(\opn{Tel}(A^0;\mathbf{a}),M))
\]
is an isomorphism. Applying Theorem \ref{thm-rgamma}, it is enough to show that the map
\[
1_{\opn{K}^{\vee}_{\infty}(A^0; \mathbf{a})}\otimes_{A^0} \opn{Hom}_{A^0}(u_{\mathbf{a}},1_M) 
: \opn{K}^{\vee}_{\infty}(A^0; \mathbf{a})\otimes_{A^0}  M
\to  \opn{K}^{\vee}_{\infty}(A^0; \mathbf{a})\otimes_{A^0} \opn{Hom}_{A^0}(\opn{Tel}(A^0;\mathbf{a}),M)
\]
is an $A$-linear isomorphism. Note that $\opn{K}^{\vee}_{\infty}(A^0; \mathbf{a})$ is a K-flat complex over $A^0$, so we may replace it by the K-flat $A^0$-complex $\opn{Tel}(A^0;\mathbf{a})$ which is quasi-isomorphic to it. Hence, it is enough to show that the map 
\[
1_{\opn{Tel}(A^0;\mathbf{a})}\otimes_{A^0} \opn{Hom}_{A^0}(u_{\mathbf{a}},1_M) 
: \opn{Tel}(A^0;\mathbf{a})\otimes_{A^0}  M
\to  \opn{Tel}(A^0;\mathbf{a})\otimes_{A^0} \opn{Hom}_{A^0}(\opn{Tel}(A^0;\mathbf{a}),M)
\]
is an isomorphism. However, in \cite{PSY1}, Lemma 6.6, it was shown that the map
\[
Q_A(1_{\opn{Tel}(A^0;\mathbf{a})}\otimes_{A^0} \opn{Hom}_{A^0}(u_{\mathbf{a}},1_M)) = 
1_{\opn{Tel}(A^0;\mathbf{a})}
\otimes_{A^0} \opn{Hom}_{A^0}(u_{\mathbf{a}},1_{Q_A(M)})
\]
is a quasi-isomorphism. Hence, by Remark \ref{remark-lifting-quasi}, the map 
\[
1_{\opn{Tel}(A^0;\mathbf{a})}\otimes_{A^0} \opn{Hom}_{A^0}(u_{\mathbf{a}},1_M)
\]
is also a quasi-isomorphism, which proves the claim.
\end{proof}

Similarly, and using \cite{PSY1}, Lemma 6.2, one has:
\begin{thm}\label{thm-LLambdaRGamma}
Let $(A,\a)$ be an internally flat weakly proregular preadic DG-algebra. Then for any $M\in \widetilde{\mrm{D}}(\opn{DGMod} A)$, the morphism
\[
\mrm{L}\Lambda_{\a}(\sigma^R_M) : \mrm{L}\Lambda_{\a} (\mrm{R}\Gamma_{\a}(M)) \to \mrm{L}\Lambda_{\a}(M))
\]
is an isomorphism in $\widetilde{\mrm{D}}(\opn{DGMod} A)$.
\end{thm}

\begin{cor}\label{cor-DG-GM}
Let $(A,\a)$ be an internally flat weakly proregular preadic DG-algebra. Then for any $M,N\in \widetilde{\mrm{D}}(\opn{DGMod} A)$, the morphisms
\[ \begin{aligned}
& \mrm{R}\opn{Hom}_A \bigl( \mrm{R} \Gamma_{\a} (M), \mrm{R} \Gamma_{\a} (N) \bigr)
\xar{\mrm{R}\opn{Hom}(1, \sigma^{\mrm{R}}_N)}
\mrm{R}\opn{Hom}_A \bigl( \mrm{R} \Gamma_{\a} (M), N \bigr) 
\\
& \qquad\xar{\mrm{R}\opn{Hom}(1, \tau^{\mrm{L}}_N)}
\mrm{R}\opn{Hom}_A \bigl( \mrm{R} \Gamma_{\a} (M), \mrm{L} \Lambda_{\a} (N) \bigr) 
\xleftarrow{\mrm{R}\opn{Hom}(\sigma^{\mrm{R}}_M, 1)}
\\ & \qquad  \qquad
\mrm{R}\opn{Hom}_A \bigl( M, \mrm{L} \Lambda_{\a} (N) \bigr)
\xleftarrow{\mrm{R}\opn{Hom}(\tau^{\mrm{L}}_M, 1)}
\mrm{R}\opn{Hom}_A \bigl( \mrm{L} \Lambda_{\a} (M), \mrm{L} \Lambda_{\a} (N) \bigr)
\end{aligned} \]
in $\widetilde{\mrm{D}}(\opn{DGMod} A)$ are isomorphisms.
\end{cor}
\begin{proof}
As shown in \cite{PSY1}, Theorem 6.12, this follows immediately from Theorems \ref{thm-rgamma}, \ref{thm-llambda}, \ref{thm-RgammaLLambda}, \ref{thm-LLambdaRGamma}, the hom-tensor adjunction, and the fact that 
$\opn{Tel}(A^0;\mathbf{a}) \otimes_{A^0} \opn{Tel}(A^0;\mathbf{a})$ is homotopy equivalent to $\opn{Tel}(A^0;\mathbf{a})$.
\end{proof}

\begin{dfn}
Let $(A,\a)$ be a preadic DG-algebra. 
\begin{enumerate}
\item A DG-module $M$ is called cohomologically $\a$-torsion if the map $\sigma^R_M:\mrm{R}\Gamma_{\a} M \to M$ is an isomorphism in $\widetilde{\mrm{D}}(\opn{DGMod} A)$. The collection of all cohomologically $\a$-torsion DG-modules is a full triangulated subcategory of 
$\widetilde{\mrm{D}}(\opn{DGMod} A)$
denoted by $\widetilde{\mrm{D}}(\opn{DGMod} A)_{\opn{\a-tor}}$.
\item A DG-module $M$ is called cohomologically $\a$-adically complete if the map $\tau^L_M:M \to \mrm{L}\Lambda_{\a} M$ is an isomorphism in $\widetilde{\mrm{D}}(\opn{DGMod} A)$. The collection of all cohomologically $\a$-adically complete DG-modules is a full triangulated subcategory of 
$\widetilde{\mrm{D}}(\opn{DGMod} A)$
denoted by $\widetilde{\mrm{D}}(\opn{DGMod} A)_{\opn{\a-com}}$.
\end{enumerate}
\end{dfn}

\begin{rem}
The theory developed in this section shows that if $(A,\a)$ is an internally flat weakly proregular preadic DG-algebra, then, as in the case of rings, the functor
\[
\mrm{R}\Gamma_{\a}: \widetilde{\mrm{D}}(\opn{DGMod} A)_{\opn{\a-com}} \to \widetilde{\mrm{D}}(\opn{DGMod} A)_{\opn{\a-tor}}
\]
is an equivalence, with quasi-inverse $\mrm{L}\Lambda_{\a}$. We refer to this equivalence as the MGM (Matlis-Greenlees-May) equivalence. Furthermore, by Corollaries \ref{cor-rgamma-id} and \ref{cor-llambda-id}, these categories are equal to the essential images of the functors $\mrm{R}\Gamma_{\a}$ and $\mrm{L}\Lambda_{\a}$ respectively.
\end{rem}

\begin{rem}
With Theorems \ref{thm-rgamma} and \ref{thm-llambda} in view, we see that this equivalence is an example of a \textbf{Generalized Foxby equivalence}, in the sense of \cite{FJ}, Section 1.5.
\end{rem}

\begin{rem}\label{rem-koszul-dg}
Let $(A,\a)$ be an internally flat weakly proregular preadic DG-algebra. Let $\mathbf{a}$ be a finite sequence in $A^0$ that generates $\a$. Following the base change formulas for the infinite dual Koszul and telescope complexes, it is natural to introduce the following notation:
\[
\opn{K}^{\vee}_{\infty}(A; \mathbf{a}) := \opn{K}^{\vee}_{\infty}(A^0;\mathbf{a}) \otimes_{A^0} A
\]
and
\[
\opn{Tel}(A;\mathbf{a}) := \opn{Tel}(A^0;\mathbf{a}) \otimes_{A^0} A.
\]
It is clear that these are K-flat and K-projective DG $A$-modules respectively.
Using this notation, the associativity of the tensor product, and the hom-tensor adjunction, we may rewrite Theorems \ref{thm-rgamma} and \ref{thm-llambda} 
and deduce that for any DG-module $M$, there are functorial isomorphisms
\[
\mrm{R}\Gamma_{\a}(M) \cong \opn{K}^{\vee}_{\infty}(A; \mathbf{a})\otimes_A M
\]
and
\[
\mrm{L}\Lambda_{\a}(M) \cong \opn{Hom}_A(\opn{Tel}(A; \mathbf{a}),M).
\]
\end{rem}

The formulas for the $\mrm{R}\Gamma_{\a}$ and $\mrm{L}\Lambda_{\a}$ functors developed above are useful when they are considered as functors $\widetilde{\mrm{D}}(\opn{DGMod} A) \to \widetilde{\mrm{D}}(\opn{DGMod} A)$. More generally, one might consider them as functors $\widetilde{\mrm{D}}(\opn{DGMod} A) \to \widetilde{\mrm{D}}(\opn{DGMod} \Lambda_{\a}(A))$. The next two propositions gives (less explicit) formulas for these functors valid over $\Lambda_{\a}(A)$. As far as we know, they are new even in the case where $A=A^0$ is a ring.

\begin{prop}\label{prop-complete-rgamma}
Let $(A,\a)$ be an internally flat weakly proregular preadic DG-algebra. Set $\widehat{A} = \Lambda_{\a}(A)$. Let $\mathbf{a}$ be a finite sequence of elements of $A^0$ that generates $\a$, and let $\widehat{\mathbf{a}}$ be its image in $\widehat{A}^0$. Then there is an isomorphism
\[
\mrm{R}\Gamma_{\a} (-) \to (\opn{K}^{\vee}_{\infty}(\widehat{A}; \widehat{\mathbf{a}}) \otimes^{\mrm{L}}_A -)
\]
of functors
\[
\widetilde{\mrm{D}}(\opn{DGMod} A) \to \widetilde{\mrm{D}}(\opn{DGMod} \widehat{A}).
\]
\end{prop}

Before the proof, we need a lemma:
\begin{lem}\label{lem-tortens}
Let $(A,\a)$ be an internally flat weakly proregular preadic DG-algebra. Set $\widehat{A} = \Lambda_{\a}(A)$. Let $Q:\widetilde{\mrm{D}}(\opn{DGMod} \widehat{A}) \to \widetilde{\mrm{D}}(\opn{DGMod} A)$ be the forgetful functor. Then for any $M \in \widetilde{\mrm{D}}(\opn{DGMod} \widehat{A})$, such that $Q(M) \in \widetilde{\mrm{D}}(\opn{DGMod} A)_{\opn{\a-tor}}$, there is a functorial isomorphism
\[
\widehat{A} \otimes^{\mrm{L}}_A M \cong M
\]
in $\widetilde{\mrm{D}}(\opn{DGMod} \widehat{A})$.
\end{lem}
\begin{proof}
Let $\mathbf{a}$ be a finite sequence in $A^0$ that generates $\a$.
By Theorem \ref{thm-rgamma}, we see that the functorial map
\[
e_{\mathbf{a},M} : \opn{K}^{\vee}_{\infty}(A; \mathbf{a}) \otimes_A M \to M
\]
is a $A$-linear quasi-isomorphism. Note that both of these DG-modules have the structure of $\widehat{A}$-DG modules, so it is easy to verify that the map $e_{\mathbf{a},M}$ constructed above is a $\widehat{A}$-linear quasi-isomorphism.
It follows that there is a sequence of functorial isomorphisms
\[
\widehat{A} \otimes^{\mrm{L}}_A M \cong \widehat{A} \otimes^{\mrm{L}}_A (\opn{K}^{\vee}_{\infty}(A; \mathbf{a}) \otimes_A M) \cong 
 (\opn{K}^{\vee}_{\infty}(A; \mathbf{a}) \otimes^{\mrm{L}}_A \widehat{A}) \otimes^{\mrm{L}}_A M
\]
in $\widetilde{\mrm{D}}(\opn{DGMod} \widehat{A})$.
Since $\widehat{A} \cong \mrm{L}\Lambda_{\a}(A)$, it follows from Theorems \ref{thm-rgamma} and \ref{thm-RgammaLLambda}, that there is an isomorphism
\[
(\opn{K}^{\vee}_{\infty}(A; \mathbf{a}) \otimes^{\mrm{L}}_A \widehat{A}) \cong \opn{K}^{\vee}_{\infty}(A; \mathbf{a})
\]
in $\widetilde{\mrm{D}}(\opn{DGMod} A)$. 
Hence, there is a functorial isomorphism
\[
(\opn{K}^{\vee}_{\infty}(A; \mathbf{a}) \otimes^{\mrm{L}}_A \widehat{A}) \otimes^{\mrm{L}}_A M \cong 
\opn{K}^{\vee}_{\infty}(A; \mathbf{a})  \otimes^{\mrm{L}}_A M.
\]
The $\widehat{A}$-linear isomorphism $e_{\mathbf{a},M}$ from above now establishes the claim.
\end{proof}

Note that if $A$ was a noetherian ring, then this lemma would have followed from the fact that $\widehat{A}\otimes_A \Gamma_{\a} (M) \cong \Gamma_{\a} (M)$ for any $A$-module $M$. The point of this lemma is that the derived version of this fact is true even when we do not know that $\widehat{A}$ is flat over $A$.

We now prove Proposition \ref{prop-complete-rgamma}:
\begin{proof} 
Let $M \in \widetilde{\mrm{D}}(\opn{DGMod} A)$. According to the above lemma, there is a functorial isomorphism
\[
\mrm{R}\Gamma_{\a} (M) \cong \widehat{A} \otimes^{\mrm{L}}_A \mrm{R}\Gamma_{\a} (M)  
\]
in $\widetilde{\mrm{D}}(\opn{DGMod} \widehat{A})$.
By Theorem \ref{thm-rgamma}, there is an $A$-linear isomorphism
\[
\mrm{R}\Gamma_{\a} (M) \cong \opn{K}^{\vee}_{\infty}(A; \mathbf{a}) \otimes_A M.
\]
Hence, there is an $\widehat{A}$-linear isomorphism
\[
\mrm{R}\Gamma_{\a} (M) \cong \widehat{A} \otimes^{\mrm{L}}_A (\opn{K}^{\vee}_{\infty}(A; \mathbf{a}) \otimes_A M)
\]
which proves the claim.
\end{proof}

Again, in this  proposition, we must use derived tensor product, because we do not know that $\widehat{A}$ is K-flat over $A$, even if $A=A^0$ is a ring, because we do now know if the map $A\to \widehat{A}$ is flat when $A$ is not noetherian. 

In a similar manner, we have:
\begin{lem}\label{lem-comphom}
Let $(A,\a)$ be an internally flat weakly proregular preadic DG-algebra. Set $\widehat{A} = \Lambda_{\a}(A)$. Let $Q:\widetilde{\mrm{D}}(\opn{DGMod} \widehat{A}) \to \widetilde{\mrm{D}}(\opn{DGMod} A)$ be the forgetful functor. Then for any $M \in \widetilde{\mrm{D}}(\opn{DGMod} \widehat{A})$, such that $Q(M) \in \widetilde{\mrm{D}}(\opn{DGMod} A)_{\opn{\a-com}}$, there is a functorial isomorphism
\[
\mrm{R}\opn{Hom}_A(\widehat{A},M) \cong M
\]
in $\widetilde{\mrm{D}}(\opn{DGMod} \widehat{A})$.
\end{lem}
\begin{proof}
Let $\mathbf{a}$ be a finite sequence in $A^0$ that generates $\a$.
By Theorem \ref{thm-llambda}, the map
\[
\opn{Hom}_A(u_{\mathbf{a}},1_M) : M \to \opn{Hom}_A(\opn{Tel}(A;\mathbf{a}),M)
\]
is an $A$-linear quasi-isomorphism, and hence, an $\widehat{A}$-linear quasi-isomorphism. It follows from the derived tensor-hom adjunction that there is a sequence of isomorphisms
\[
\mrm{R}\opn{Hom}_A(\widehat{A},M) \cong \mrm{R}\opn{Hom}_A(\widehat{A},\opn{Hom}_A(\opn{Tel}(A;\mathbf{a}),M)) \cong \mrm{R}\opn{Hom}_A(\widehat{A} \otimes_A \opn{Tel}(A;\mathbf{a}),M)
\]
in $\widetilde{\mrm{D}}(\opn{DGMod} \widehat{A})$.
As in the proof of the previous lemma, there is an $A$-linear isomorphism
\[
\widehat{A} \otimes_A \opn{Tel}(A;\mathbf{a}) \cong \opn{Tel}(A;\mathbf{a}),
\]
and hence, there is an isomorphism
\[
\mrm{R}\opn{Hom}_A(\widehat{A} \otimes_A \opn{Tel}(A;\mathbf{a}),M) \cong 
\mrm{R}\opn{Hom}_A( \opn{Tel}(A;\mathbf{a}),M)
\]
in $\widetilde{\mrm{D}}(\opn{DGMod} \widehat{A})$.
The $\widehat{A}$-linear quasi-isomorphism $\opn{Hom}_A(u_{\mathbf{a}},1_M)$ now establishes the claim.
\end{proof}

From this lemma it follows that:
\begin{prop}\label{prop-complete-llambda}
Let $(A,\a)$ be an internally flat weakly proregular preadic DG-algebra. Set $\widehat{A} = \Lambda_{\a}(A)$. Let $\mathbf{a}$ be a finite sequence of elements of $A^0$ that generates $\a$, and let $\widehat{\mathbf{a}}$ be its image in $\widehat{A}^0$. Then there is an isomorphism
\[
\mrm{L}\Lambda_{\a} (-) \to \mrm{R}\opn{Hom}_A(\opn{Tel}(\widehat{A};\widehat{\mathbf{a}}),-)
\]
of functors
\[
\widetilde{\mrm{D}}(\opn{DGMod} A) \to \widetilde{\mrm{D}}(\opn{DGMod} \widehat{A}).
\]
\end{prop}
\begin{proof}
Similar to the proof of Proposition \ref{prop-complete-rgamma}, using Lemma \ref{lem-comphom}.
\end{proof}

It follows immediately from Propositions \ref{prop-complete-rgamma}, \ref{prop-complete-llambda} and from Theorems \ref{thm-RgammaLLambda} and \ref{thm-LLambdaRGamma} that
\begin{cor}\label{cor-complete-rgamma-llambda}
Let $(A,\a)$ be an internally flat weakly proregular preadic DG-algebra. Then the there are isomorphisms
\[
\mrm{R}\Gamma_{\a} ( -) \cong \mrm{R}\Gamma_{\a} \circ \mrm{L}\Lambda_{\a} (-)
\]
and
\[
\mrm{L}\Lambda_{\a} ( -) \cong \mrm{L}\Lambda_{\a} \circ \mrm{R}\Gamma_{\a} (-)
\]
of functors
\[
\widetilde{\mrm{D}}(\opn{DGMod} A) \to \widetilde{\mrm{D}}(\opn{DGMod} \widehat{A}).
\]
\end{cor}

Note that in the corollary, in the first isomorphism, $\mrm{R}\Gamma_{\a}$ is a functor $\widetilde{\mrm{D}}(\opn{DGMod} A) \to \widetilde{\mrm{D}}(\opn{DGMod} \widehat{A})$, while $\mrm{L}\Lambda_{\a}$ is a functor $\widetilde{\mrm{D}}(\opn{DGMod} A) \to \widetilde{\mrm{D}}(\opn{DGMod} A)$, and vice versa in the second isomorphism.

\section{Derived adic Hochschild cohomology}

In this section we discuss the main theme of this article, namely, the complete and torsion Derived Hochschild cohomology functors.

First, we recall the construction of the derived Hochschild cohomology functor from \cite{AILN}, under our additional assumption that all rings are commutative. Let $\k$ be a commutative ring, and let $A$ be a commutative $\k$-algebra. According to \cite{AILN}, Theorem 3,2, if $\k \to \widetilde{A} \to A$ is a K-flat DG-resolution of $\k \to A$, there is a functor
\[
\mrm{R}\opn{Hom}_{\widetilde{A}\otimes_{\k} \widetilde{A}} ( A, -\otimes^{\mrm{L}}_{\k} -) :\mrm{D}(\opn{Mod} A) \times \mrm{D}(\opn{Mod} A) \to \mrm{D}(\opn{Mod} A)
\]
If $\k \to \widetilde{B} \to A$ is another K-flat resolution of $\k \to A$, there is a canonical natural equivalence of functors:
\[
w^{\widetilde{A}\widetilde{B}}: \mrm{R}\opn{Hom}_{\widetilde{A}\otimes_{\k} \widetilde{A}} ( A, -\otimes^{\mrm{L}}_{\k} -) \to \mrm{R}\opn{Hom}_{\widetilde{B}\otimes_{\k} \widetilde{B}} ( A, -\otimes^{\mrm{L}}_{\k} -)
\]
Moreover, for every K-flat resolution $\k \to \widetilde{C} \to A$ of $\k \to A$, there is an equality:
\[
w^{\widetilde{A}\widetilde{C}} = w^{\widetilde{B}\widetilde{C}} \circ w^{\widetilde{A}\widetilde{B}}.
\]
With this theorem in hand, \cite{AILN} Remark 3.3 suggests the following natural notation: 
\[
\mrm{R}\opn{Hom}_{A\otimes^{\mrm{L}}_{\k} A} ( A, -\otimes^{\mrm{L}}_{\k} -)
\]
for the functor above, obtained by choosing some K-flat resolution of $\k \to A$. In the case where $A$ is a projective $\k$-module, the cohomology of the complex $\mrm{R}\opn{Hom}_{A\otimes^{\mrm{L}}_{\k} A} ( A, M\otimes^{\mrm{L}}_{\k} N)$ is the classical Hochschild cohomology 
$\mrm{HH}^n(A|\k; M\otimes^{\mrm{L}}_{\k} N)$ of the  complex $M\otimes^{\mrm{L}}_{\k} N$. This observation, made in \cite{AILN} Remark 3.3, explains the name - the Derived Hochschild cohomology functor, for this construction.

We now study variants of this construction in the adic category. 

\begin{dfn}
Let $\k$ be a commutative ring. Let $(A,\a)$ be a weakly proregular preadic $\k$-algebra.

\begin{enumerate}
\item The complete derived Hochschild cohomology functor is defined to be the functor:
\[
\mrm{L}\Lambda_{\a} (\mrm{R}\opn{Hom}_{A\otimes^{\mrm{L}}_{\k} A} ( A, -\otimes^{\mrm{L}}_{\k} -) )  : \mrm{D}(\opn{Mod} A) \times \mrm{D}(\opn{Mod} A) \to \mrm{D}(\opn{Mod} A)_{\opn{\a-com}}
\]
\item The torsion derived Hochschild cohomology functor is defined to be the functor:
\[
\mrm{R}\Gamma_{\a} (\mrm{R}\opn{Hom}_{A\otimes^{\mrm{L}}_{\k} A} ( A, -\otimes^{\mrm{L}}_{\k} -) )  : \mrm{D}(\opn{Mod} A) \times \mrm{D}(\opn{Mod} A) \to \mrm{D}(\opn{Mod} A)_{\opn{\a-tor}}
\]
\end{enumerate}
\end{dfn}
\begin{rem}
It is clear from the definition and the MGM equivalence (Theorem \ref{thm-ring-mgm}(2)) that these two functors are related via the following bifunctorial isomorphisms:
\[
\mrm{L}\Lambda_{\a} (\mrm{R}\Gamma_{\a} (\mrm{R}\opn{Hom}_{A\otimes^{\mrm{L}}_{\k} A} ( A, -\otimes^{\mrm{L}}_{\k} -) ) ) \to \mrm{L}\Lambda_{\a}  (\mrm{R}\opn{Hom}_{A\otimes^{\mrm{L}}_{\k} A} ( A, -\otimes^{\mrm{L}}_{\k} -) ) 
\]
and
\[
\mrm{R}\Gamma_{\a} (\mrm{R}\opn{Hom}_{A\otimes^{\mrm{L}}_{\k} A} ( A, -\otimes^{\mrm{L}}_{\k} -) )  \to
\mrm{R}\Gamma_{\a} ( \mrm{L}\Lambda_{\a} (\mrm{R}\opn{Hom}_{A\otimes^{\mrm{L}}_{\k} A} ( A, -\otimes^{\mrm{L}}_{\k} -) ) )
\]
\end{rem}

We now turn to study these functors. We first concentrate on the complete case.

\begin{lem}\label{lem-llamofhom}
Suppose $(B,\b)$ and $(C,\c)$ are two internally flat weakly proregular preadic DG-algebras. Assume $f:B\to C$ is a preadic DG-algebra map, and assume further that $\c = f(\b)\cdot B^0 \subseteq C^0$. Then there is an isomorphism
\[
\mrm{L}\Lambda_{\c} ( \mrm{R}\opn{Hom}_B(C,- ) ) \cong \mrm{R}\opn{Hom}_B (C,\mrm{L}\Lambda_{\b}(-))
\]
of functors
\[
\widetilde{\mrm{D}}(\opn{DGMod} B) \to \widetilde{\mrm{D}}(\opn{DGMod} C). 
\]
\end{lem}
\begin{proof}
Let $\mathbf{b}$ be a finite sequence in $B^0$ that generates $\b$. Let $\mathbf{c} \subseteq C^0$ be the image of $\mathbf{b}$ under $f$. For each $M \in \widetilde{\mrm{D}}(\opn{DGMod} B)$, by Theorem \ref{thm-llambda}, there is a functorial isomorphism
\[
\opn{tel}^{\mrm{L}}_{\mathbf{a},\mrm{R}\opn{Hom}_B(C,M)}: \opn{Hom}_{C^0}(\opn{Tel}(C^0;\mathbf{c}),\mrm{R}\opn{Hom}_B(C,M)) \to \mrm{L}\Lambda_{\c} (\mrm{R}\opn{Hom}_B(C,M))
\]
in $\widetilde{\mrm{D}}(\opn{DGMod} C)$. Since $\opn{Tel}(C^0;\mathbf{c})$ is K-projective over $C^0$, by the derived hom-tensor adjunction, there is a functorial isomorphism
\[
\opn{Hom}_{C^0}(\opn{Tel}(C^0;\mathbf{c}),\mrm{R}\opn{Hom}_B(C,M)) \to 
\mrm{R}\opn{Hom}_B( \opn{Tel}(C^0;\mathbf{c}) \otimes_{C^0} C, M)
\]
in $\widetilde{\mrm{D}}(\opn{DGMod} C)$. 
Since $f(\mathbf{b}) = \mathbf{c}$, there is an isomorphism of complexes
$\opn{Tel}(C^0;\mathbf{c}) \cong \opn{Tel}(B^0;\mathbf{b}) \otimes_{B^0} C^0$.
Hence, using the derived hom-tensor adjunction twice, we see that there are functorial isomorphisms in $\widetilde{\mrm{D}}(\opn{DGMod} C)$:
\[
\begin{aligned}
\mrm{R}\opn{Hom}_B( \opn{Tel}(C^0;\mathbf{c}) \otimes_{C^0} C, M) \cong\\ 
\mrm{R}\opn{Hom}_B( (\opn{Tel}(B^0;\mathbf{b}) \otimes_{B^0} B) \otimes_B C , M) \cong\\ 
\mrm{R}\opn{Hom}_B(C,\opn{Hom}_{B^0}( \opn{Tel}(B^0;\mathbf{b}), M)).
\end{aligned}
\]
Using Theorem \ref{thm-llambda}, we see that there is a functorial isomorphism
\[
\opn{tel}^{\mrm{L}}_{\mathbf{b},M}:
\opn{Hom}_{B^0}( \opn{Tel}(B^0;\mathbf{b}), M) \to \mrm{L}\Lambda_{\b}(M)
\] 
Composing all the above maps gives the required functorial isomorphism.
\end{proof}

\begin{rem}
In the case where $B=C=B^0$, this lemma was proved in \cite{AJL}.
\end{rem}

\begin{lem}\label{lemma-rhom-with-completion}
Let $(B,\b)$ be an internally flat weakly proregular preadic DG-algebra. Set $\widehat{B} = \Lambda_{\b}(B)$. Let $Q:\widetilde{\mrm{D}}(\opn{DGMod} \widehat{B}) \to \widetilde{\mrm{D}}(\opn{DGMod} B)$ be the forgetful functor. Let $(C,\c)$ be a preadic DG-algebra, and let $\widehat{B} \to C$ be a preadic DG-algebra map. Then for any $M \in \widetilde{\mrm{D}}(\opn{DGMod} \widehat{B})$, such that $Q(M) \in \widetilde{\mrm{D}}(\opn{DGMod} B)_{\opn{\b-com}}$, there is a functorial isomorphism
\[
\mrm{R}\opn{Hom}_B(C,M) \cong \mrm{R}\opn{Hom}_{\widehat{B}}(C,M)
\]
in $\widetilde{\mrm{D}}(\opn{DGMod} C)$.
\end{lem}
\begin{proof}
By the derived hom-tensor adjunction, there is a functorial isomorphism
\[
\mrm{R}\opn{Hom}_B(C,M) \cong \mrm{R}\opn{Hom}_{\widehat{B}}(C,\mrm{R}\opn{Hom}_B(\widehat{B},M))
\]
in $\widetilde{\mrm{D}}(\opn{DGMod} C)$. 
Thus, it remains to show that there is a functorial isomorphism 
\[
\mrm{R}\opn{Hom}_B(\widehat{B},M) \cong M
\]
in $\widetilde{\mrm{D}}(\opn{DGMod} \widehat{B})$, but this follows from Lemma \ref{lem-comphom}.
\end{proof}

\begin{prop}\label{prop:internally-flat-tensor}
Let $\k$ be a ring, let $A$ be an internally flat DG-algebra, and assume that $A$ is K-flat over $\k$. Then the DG-algebra $A\otimes_{\k} A$ is also internally flat.
\end{prop}
\begin{proof}
We must show that $A\otimes_{\k} A$ is K-flat over $A^0\otimes_{\k} A^0$. Let $M$ be an acyclic complex over $A^0\otimes_{\k} A^0$. To see that $M\otimes_{A^0\otimes_{\k} A^0} A\otimes_{\k} A$ is acyclic, note that there is an $A^0$-linear isomorphism of complexes
\[
M\otimes_{A^0\otimes_{\k} A^0} A\otimes_{\k} A \cong A\otimes_{A^0} M \otimes_{A^0} A.
\]
Since $A$ is K-flat over $A^0$, the result follows.
\end{proof}

Here is the main result of this section.

\begin{thm}\label{thm-main-c-formula}
Let $\k$ be a commutative ring. Let $(A,\a)$ be an adic ring which is a weakly proregular preadic $\k$-algebra. Let $\k \to \widetilde{A} \to A$ be a K-flat resolution of $\k \to A$ which is internally flat. Let $I\subseteq (\widetilde{A}\otimes_{\k} \widetilde{A})^0$ be a weakly proregular ideal, such that its image under the 
composed map 
\[
(\widetilde{A}\otimes_{\k} \widetilde{A})^0 \to \widetilde{A}^0 \to A
\]
is equal to $\a$. Then there is an isomorphism 
\[
v^{\widetilde{A},I}:
\mrm{L}\Lambda_{\a} (\mrm{R}\opn{Hom}_{A\otimes^{\mrm{L}}_{\k} A} ( A, -\otimes^{\mrm{L}}_{\k} -) )  
\cong 
\mrm{R}\opn{Hom}_{\Lambda_I(\widetilde{A}\otimes_{\k} \widetilde{A})} (A, \mrm{L}\Lambda_I ( - \otimes^{\mrm{L}}_{\k} - ) )
\]
of functors
\[
\mrm{D}(\opn{Mod} A) \times \mrm{D}(\opn{Mod} A) \to \mrm{D}(\opn{Mod} A)_{\opn{\a-com}}
\]
\end{thm}
\begin{proof}
Note that by Proposition \ref{prop:internally-flat-tensor}, the DG-algebra $\widetilde{A}\otimes_{\k} \widetilde{A}$ is internally flat.
Let $M,N \in \mrm{D}(\opn{Mod} A)$. By the definition of the derived Hochschild cohomology functor, there is a functorial isomorphism in $\mrm{D}(\opn{Mod} A)_{\opn{\a-com}}$:
\[
\mrm{L}\Lambda_{\a} (\mrm{R}\opn{Hom}_{A\otimes^{\mrm{L}}_{\k} A} ( A, M\otimes^{\mrm{L}}_{\k} N) )  
\cong 
\mrm{L}\Lambda_{\a} (\mrm{R}\opn{Hom}_{\widetilde{A}\otimes_{\k} \widetilde{A}} ( A, M \otimes^{\mrm{L}}_{\k} N) )  
\]
According to Lemma \ref{lem-llamofhom}, there is an isomorphism of functors
\[
\mrm{L}\Lambda_{\a} (\mrm{R}\opn{Hom}_{\widetilde{A}\otimes_{\k} \widetilde{A}} ( A, M \otimes^{\mrm{L}}_{\k} N) )  
\cong
\mrm{R}\opn{Hom}_{\widetilde{A}\otimes_{\k} \widetilde{A}} ( A, \mrm{L}\Lambda_I(M \otimes^{\mrm{L}}_{\k} N)  ).   
\]
The assumption that $I \cdot A = \a$, and the fact that $A$ is $\a$-adically complete, ensures that $A$ is a DG $\Lambda_I(\widetilde{A} \otimes_{\k} \widetilde{A})$-module. Since $I$ is weakly proregular, the DG-module $\mrm{L}\Lambda_I(M \otimes^{\mrm{L}}_{\k} N)$ is cohomologically $I$-adically complete, so that by Lemma \ref{lemma-rhom-with-completion}, there is a functorial isomorphism
\[
\mrm{R}\opn{Hom}_{\widetilde{A}\otimes_{\k} \widetilde{A}} ( A, \mrm{L}\Lambda_I(M \otimes^{\mrm{L}}_{\k} N)  ) \cong
\mrm{R}\opn{Hom}_{\Lambda_I(\widetilde{A}\otimes_{\k} \widetilde{A})} ( A, \mrm{L}\Lambda_I(M \otimes^{\mrm{L}}_{\k} N)  ).   
\]
Composing these three isomorphisms, we obtain the required functorial isomorphism $v^{\widetilde{A},I}$.
\end{proof}

If one is only interested in the functor in the right hand side, we immediately get the following independence result, following \cite{AILN}, Theorem 3.2:

\begin{cor}\label{cor-ind-c}
Let $\k$ be a commutative ring. Let $(A,\a)$ be an adic ring which is a weakly proregular preadic $\k$-algebra. For any K-flat resolution $\k \to \widetilde{A} \to A$ of $\k \to A$ which is internally flat, and for any weakly proregular ideal $I\subseteq (\widetilde{A}\otimes_{\k} \widetilde{A})^0$, such that its image under the 
composed map 
\[
(\widetilde{A}\otimes_{\k} \widetilde{A})^0 \to \widetilde{A}^0 \to A
\]
is equal to $\a$, there is a functor
\[
\mrm{R}\opn{Hom}_{\Lambda_I(\widetilde{A}\otimes_{\k} \widetilde{A})} (A, \mrm{L}\Lambda_I ( - \otimes^{\mrm{L}}_{\k} - ) ) :\mrm{D}(\opn{Mod} A) \times \mrm{D}(\opn{Mod} A) \to \mrm{D}(\opn{Mod} A)_{\opn{\a-com}}.
\]
For every K-flat resolution $\k \to \widetilde{B} \to A$ of $k \to A$, which is internally flat, and for every weakly proregular ideal $I'\subseteq (\widetilde{B}\otimes_{\k} \widetilde{B})^0$, such that its image under the 
composed map 
\[
(\widetilde{B}\otimes_{\k} \widetilde{B})^0 \to \widetilde{B}^0 \to A
\]
is equal to $\a$, there is a canonical isomorphism of functors
\[
w^{(\widetilde{A},I),(\widetilde{B},I')} :
\mrm{R}\opn{Hom}_{\Lambda_I(\widetilde{A}\otimes_{\k} \widetilde{A})} (A, \mrm{L}\Lambda_I ( - \otimes^{\mrm{L}}_{\k} - ) ) \to 
\mrm{R}\opn{Hom}_{\Lambda_{I'}(\widetilde{B}\otimes_{\k} \widetilde{B})} (A, \mrm{L}\Lambda_{I'} ( - \otimes^{\mrm{L}}_{\k} - ) ). 
\]
Furthermore, if $(\widetilde{A},I)$, $(\widetilde{B},I')$ and $(\widetilde{C},I'')$ are three such pairs, then there is an equality
\[
w^{(\widetilde{A},I),(\widetilde{C},I'')} = 
w^{(\widetilde{B},I'),(\widetilde{C},I'')} \circ
w^{(\widetilde{A},I),(\widetilde{B},I')}
\]
\end{cor}
\begin{proof}
Define 
\[
w^{(\widetilde{A},I),(\widetilde{B},I')} := v^{(\widetilde{B},I')} \circ \mrm{L}\Lambda_{\a} ( w^{\widetilde{A} \widetilde{B}} ) \circ (v^{(\widetilde{A},I)})^{-1}.
\]
This is a composition of functorial isomorphisms, so it is also a functorial isomorphism. Since by \cite{AILN}, Theorem 3.2, there is an equality
\[
w^{\widetilde{A}\widetilde{C}} = w^{\widetilde{B}\widetilde{C}} \circ w^{\widetilde{A}\widetilde{B}}
\]
and since $\mrm{L}\Lambda_{\a}$ is a functor, we get the equality
\[
w^{(\widetilde{A},I),(\widetilde{C},I'')} = 
w^{(\widetilde{B},I'),(\widetilde{C},I'')} \circ
w^{(\widetilde{A},I),(\widetilde{B},I')}.
\]
\end{proof}

In the noetherian case, we can drop the word weakly proregular from the statement, and obtain the following:

\begin{cor}\label{cor-cformula-noetherian}
Let $\k$ be a noetherian ring. Let $(A,\a)$ be an adic ring which is an essentially formally of finite type preadic $\k$-algebra. Then for any adic K-flat resolution $\k \to (\widetilde{A},\widetilde{\a}) \to (A,\a)$ of $\k \to (A,\a)$, such that $(\widetilde{A}^0,\widetilde{\a})$ is essentially formally of finite type over $\k$, 
there is an isomorphism 
\[
\mrm{L}\Lambda_{\a} (\mrm{R}\opn{Hom}_{A\otimes^{\mrm{L}}_{\k} A} ( A, -\otimes^{\mrm{L}}_{\k} -) )  
\cong 
\mrm{R}\opn{Hom}_{\Lambda_{\widetilde{\a}^e} (\widetilde{A}\otimes_{\k} \widetilde{A})} (A, \mrm{L}\Lambda_{\widetilde{\a}^e} ( - \otimes^{\mrm{L}}_{\k} - ) )
\]
of functors
\[
\mrm{D}(\opn{Mod} A) \times \mrm{D}(\opn{Mod} A) \to \mrm{D}(\opn{Mod} A)_{\opn{\a-com}}
\]
where $\widetilde{\a}^e = \widetilde{\a} \otimes_{\k} \widetilde{A}^0 + \widetilde{A}^0 \otimes_{\k} \widetilde{\a}$ is the canonical ideal of definition of $\widetilde{A}^0 \otimes_{\k} \widetilde{A}^0$. Moreover, the ring $\Lambda_{\widetilde{\a}^e} (\widetilde{A}\otimes_{\k} \widetilde{A})^0$ is noetherian. 
\end{cor}
\begin{proof}
Since $\k \to \widetilde{A}^0$ is flat, and essentially formally of finite type, by Corollary \ref{cor-wpr-of-efft}, we see that the ideal $\widetilde{\a}^e$ is weakly proregular, so the above is reduced to the statement of the theorem. The last claim follows from the fact that the ring $\widetilde{A} /\widetilde{\a} \otimes_{\k} \widetilde{A}/\widetilde{\a}$ is noetherian.
\end{proof}

\begin{rem}
Suppose $\k$ is a noetherian ring. Let $(A,\a)$ be an adic ring which is an essentially formally of finite type preadic $\k$-algebra. Then by Proposition \ref{prop-existence-of-res}(2), there is an adic K-flat resolution $(\widetilde{A},\widetilde{\a})$ which satisfy the assumptions of this corollary.
\end{rem}

\begin{rem}\label{rem-ae-is-wpr}
If $\k$ is a field, one can drop the finiteness assumption from $A$, and just assume that $A$ is noetherian. This is because, by \cite{PSY1}, Example 3.35, in that case the ideal $\a \otimes_{\k} A + A\otimes_{\k} \a$ is weakly proregular.
\end{rem}

\begin{rem}\label{rem-main-c-bimodule}
The proof of Theorem \ref{thm-main-c-formula} actually say more: Let $\k$ be  a commutative ring, and let $(A,\a)$ be an internally flat weakly proregular adic DG-algebra. Then for any weakly proregular ideal $I\subseteq A^0\otimes_{\k} A^0$, such that $I\cdot A^0 = \a$, there is an isomorphism
\[
\mrm{L}\Lambda_{\a} (\mrm{R}\opn{Hom}_{A\otimes_{\k} A} (A,-) ) \cong \mrm{R}\opn{Hom}_{\Lambda_I(A\otimes_{\k} A)}(A,\mrm{L}\Lambda_I(- ) )
\]
of functors
\[
\widetilde{\mrm{D}}(\opn{DGMod} (A\otimes_{\k} A)) \to \widetilde{\mrm{D}}(\opn{DGMod} A).
\]
Here, the left hand side is the derived completion of the (ordinary) Hochschild cohomology of the DG-algebra $A$ over $\k$, while the right hand side should be thought of as an adic Hochschild cohomology of the DG-algebra $A$ over $\k$. As far as we know, this result is new even in the case where $A=A^0$ is a commutative ring. In that case, assuming $A$ is flat over $\k$ the adic Hochschild cohomology functor
\[
\mrm{R}\opn{Hom}_{\widehat{A\otimes_{\k} A}}(A,-)
\]
was defined in \cite{HU}.
\end{rem}

\begin{lem}\label{lem-adic-cc}
Let $(A,\a)$ and $(B,\b)$ be two weakly proregular preadic rings. Suppose that $A\to B$ is a preadic ring map such that $\b = \a\cdot B$. Let $Q:\mrm{D}(\opn{Mod} B) \to \mrm{D}(\opn{Mod} A)$ be the forgetful functor. Let $M \in \mrm{D}(\opn{Mod} B)$ be such that $Q(M)$ is cohomologically $\a$-adically complete. Then $M$ is cohomologically $\b$-adically complete. 
\end{lem}
\begin{proof}
Let $\mathbf{a}$ be a finite sequence that generates $\a$, and let $\mathbf{b}$ be the image of $\a$ in $B$.
Since $Q(M)$ is cohomologically $\a$-adically complete, by Theorem \ref{thm-llambda}, the natural map 
\[
Q(M) \to \opn{Hom}_A(\opn{Tel}(A;\mathbf{a}),Q(M))
\]
is a quasi-isomorphism. 
The commutative diagram
\[
\xymatrix{
M \ar[r]\ar[d] & \opn{Hom}_B(\opn{Tel}(B;\mathbf{b}),M) \ar[d]\\
Q(M) \ar[r] &\opn{Hom}_A(\opn{Tel}(A;\mathbf{a}),Q(M))
}
\]
now establishes the claim.
\end{proof}

Let $\k$ be a field, let $A$ be a $\k$-algebra, and let $M$ be an $A$-module. We denote by $\mrm{HH}^n(A|\k;M)$ the $n$-th Hochschild cohomology of $A$ with coefficients in $M$. If $A$ is equipped with an adic topology, we denote by $\widehat{\mrm{HH}}^n(A|\k;M)$ the $n$-th adic Hochschild cohomology of $A$ with coefficients in $M$. By definition 
\[
\widehat{\mrm{HH}}^n(A|\k;M) := \opn{Ext}^n_{\widehat{A\otimes_{\k} A}}(A,M).
\]
\begin{cor}\label{cor-comp-thm}
Let $\k$ be a field. Let $A$ be a $\k$-algebra. Assume that $A$ is noetherian and $\a$-adically complete with respect to some ideal $\a\subseteq A$. Let $M$ be a finitely generated $A$-module (or more generally, any bounded complex with $\a$-adically complete cohomologies). Then for all $n$, there is an isomorphism
\[
\mrm{HH}^n(A|\k;M) \cong \widehat{\mrm{HH}}^n(A|\k;M).
\]
\end{cor}
\begin{proof}
According to Remark \ref{rem-ae-is-wpr}, the ideal $\a^e = \a\otimes_{\k} A + A\otimes_{\k} \a$ is weakly proregular. Hence, using the variation of Theorem \ref{thm-main-c-formula} noted in Remark \ref{rem-main-c-bimodule}, we see that there is an isomorphism
\[
\mrm{L}\Lambda_{\a} \mrm{R}\opn{Hom}_{A\otimes_{\k} A}(A,M) \cong \mrm{R}\opn{Hom}_{\widehat{A\otimes_{\k} A}} (A,\mrm{L}\Lambda_{\a^e} (M)).
\]
By Lemma \ref{lem-llamofhom}, there is an isomorphism
\[
\mrm{L}\Lambda_{\a} \mrm{R}\opn{Hom}_{A\otimes_{\k} A}(A,M) \cong \mrm{R}\opn{Hom}_{A\otimes_{\k} A}(A,\mrm{L}\Lambda_{\a^e}(M)).
\]
Since $A$ is a complete noetherian ring, it follows that any finitely generated $A$-module is $\a$-adically complete. According to \cite{PSY2}, Theorem 1.21, if $M$ is a bounded complex with $\a$-adically complete cohomologies, $M$ is cohomologically $\a$-adically complete. Hence, by Lemma \ref{lem-adic-cc}, $M$, considered as an object of $\mrm{D}(\opn{Mod} A\otimes_{\k} A)$ is also cohomologically $\a^e$-adically complete. Thus, there are isomorphisms
\[
\mrm{R}\opn{Hom}_{A\otimes_{\k} A}(A,\mrm{L}\Lambda_{\a^e}(M)) \cong \mrm{R}\opn{Hom}_{A\otimes_{\k} A}(A,M)
\]
and
\[
\mrm{R}\opn{Hom}_{\widehat{A\otimes_{\k} A}}(A,\mrm{L}\Lambda_{\a^e}(M)) \cong \mrm{R}\opn{Hom}_{\widehat{A\otimes_{\k} A}}(A,M)
\]
which proves the claim.
\end{proof}

From this, it follows immediately that:
\begin{cor}
Let $\k$ be a field. Let $A$ be a $\k$-algebra. Assume that $A$ is noetherian and $\a$-adically complete with respect to some ideal $\a\subseteq A$. Let $M$ be a finitely generated $A$-module. Suppose that $(A,\a)$ is essentially formally of finite type over $\k$ (or, more generally, that $\Lambda_{\a^e}(A\otimes_{\k} A)$ is a noetherian ring). Then for all $n$, $\mrm{HH}^n(A|\k;M)$, the $n$-th Hochschild cohomology of $A$ over $\k$ with coefficients in $M$, is a finitely generated $A$-module.
\end{cor}

\begin{exa}
Let $\k$ be a noetherian ring, and let 
\[
(A,\a) = (\k[y_1,\dots,y_m][[x_1,\dots,x_n]], (x_1,\dots,x_n))
\]
where $m\ge 0$ and $n>0$. The ring $A\otimes_{\k} A$ is usually not noetherian (this is the case for example if $\k$ is a field, and $\opn{char}\k = 0$), so it is difficult to calculate $\mrm{HH}^n(A|\k; -)$. 

Let $\a^e = (x_1\otimes_{\k} 1,\dots, x_n\otimes_{\k} 1,1\otimes_{\k} x_1,\dots,1\otimes_{\k} x_n) \subseteq A\otimes_{\k} A$. Note that $\widehat{A\otimes_{\k} A} := \Lambda_{\a^e}(A\otimes_{\k} A) \cong \k[y_1,\dots,y_{2m}][[x_1,\dots,x_{2n}]]$.
Note also that by Corollary \ref{cor-wpr-of-efft}, the ideal $\a^e$ is weakly proregular.  Since $A$ is flat over $\k$, Theorem \ref{thm-main-c-formula} says that for any pair of complexes $M,N\in \mrm{D}(A)$, there is a functorial isomorphism
\[
\begin{aligned}
\mrm{L}\Lambda_{\a} (\mrm{R}\opn{Hom}_{A \otimes_{\k} A}(A,M\otimes^{\mrm{L}}_{\k} N) \cong\\
\mrm{R}\opn{Hom}_{\widehat{A\otimes_{\k} A}}(A,\mrm{L}\Lambda_{\a^e}( M\otimes^{\mrm{L}}_{\k} N))
\end{aligned}
\]
Note that $I = \ker(\widehat{A\otimes_{\k} A} \to A)$ is generated by a regular sequence of length $m+n$. Hence, the above $\mrm{R}\opn{Hom}$ may be calculated using the Koszul resolution of $A$ over $\widehat{A\otimes_{\k} A}$. Thus, by \cite{RD}, Corollary III.7.3, there is a functorial isomorphism
\[
\mrm{R}\opn{Hom}_{\widehat{A\otimes_{\k} A}}(A,\mrm{L}\Lambda_{\a^e}( M\otimes^{\mrm{L}}_{\k} N)) \cong  \omega \otimes^{\mrm{L}}_{\widehat{A\otimes_{\k} A}} \mrm{L}\Lambda_{\a^e}( M\otimes^{\mrm{L}}_{\k} N)[-(m+n)]
\]
where $\omega = \opn{Hom}_A(\wedge^{m+n}(I/I^2),A)$. It can be shown that $\wedge^{m+n}(I/I^2) \cong \widehat{\Omega}^{m+n}_{A/\k}$. But in this particular case of a ring of formal power series over a polynomial ring, $\widehat{\Omega}^{m+n}_{A/\k} \cong A$, so we see get that
\[
\mrm{L}\Lambda_{\a} (\mrm{R}\opn{Hom}_{A\otimes_{\k} A}(A,M\otimes^{\mrm{L}}_{\k} N)) \cong
A \otimes^{\mrm{L}}_{\widehat{A\otimes_{\k} A}} \mrm{L}\Lambda_{\a^e}( M\otimes^{\mrm{L}}_{\k} N)[-(m+n)]
\]

The right hand side should be thought of as an adic Hochschild homology functor. Thus, the above isomorphism is a variation, in the adic category, of the \textbf{Van den Bergh duality} (\cite{VdB1}, Theorem 1). 
\end{exa}

\begin{rem}
In the subsequent article \cite{SH}, we generalize the above example to the case where $\k$ is a noetherian ring, and $(A,\a)$ is an essentially formally of finite type and formally smooth $\k$-algebra with connected spectrum. In this case, $\widehat{\Omega}^1_{A/\k}$ is a projective $A$-module of finite rank $n$, and one can show that 
\[
\mrm{L}\Lambda_{\a} \mrm{R}\opn{Hom}_{A\otimes_{\k} A}(A, M\otimes^{\mrm{L}}_{\k} N) \cong \opn{Hom}_A(\widehat{\Omega}^n_{A/\k},A) \otimes^{\mrm{L}}_{\widehat{A\otimes_{\k} A}} \mrm{L}\Lambda_I(M\otimes^{\mrm{L}}_{\k} N)[-n]
\]
for any pair of complexes $M,N$ which satisfy suitable finiteness conditions, where $I=\ker(A\otimes_{\k} A\to A/\a)$. Moreover, the ideal $I$ is weakly proregular.
\end{rem}

\begin{lem}\label{lem-llamofllam}
Let $\k$ be a commutative ring. Let $(A,\a)$ be an adic ring which is a weakly proregular preadic $\k$-algebra. Let $\k \to \widetilde{A} \to A$ be a K-flat resolution of $\k \to A$ which is internally flat, and let $I\subseteq (\widetilde{A}\otimes_{\k} \widetilde{A})^0$ be a weakly proregular ideal, such that its image under the 
composed map 
\[
(\widetilde{A}\otimes_{\k} \widetilde{A})^0 \to \widetilde{A}^0 \to A
\]
is equal to $\a$. Then there are isomorphisms
\[
\mrm{L}\Lambda_I( - \otimes^{\mrm{L}}_{\k} -) \cong \mrm{L}\Lambda_I( \mrm{L}\Lambda_{\a}(-) \otimes^{\mrm{L}}_{\k} -)
\]
and
\[
\mrm{L}\Lambda_I( - \otimes^{\mrm{L}}_{\k} -) \cong \mrm{L}\Lambda_I( \mrm{R}\Gamma_{\a}(-) \otimes^{\mrm{L}}_{\k} -)
\]
of functors
\[
\mrm{D}(\opn{Mod} A) \times \mrm{D}(\opn{Mod} A) \to \widetilde{\mrm{D}}(\opn{DGMod} (\Lambda_I(A\otimes_{\k} A))).
\]
\end{lem}
\begin{proof}
Let $M,N \in \mrm{D}(\opn{Mod} A)$.
For the first isomorphism, by Corollary \ref{cor-complete-rgamma-llambda}, it is enough to show that there is a functorial isomorphism
\[
\mrm{R}\Gamma_I( M \otimes^{\mrm{L}}_{\k} N) \cong \mrm{R}\Gamma_I( \mrm{L}\Lambda_{\a}(M) \otimes^{\mrm{L}}_{\k} N).
\]
Let $\widetilde{\mathbf{a}}$ be a finite sequence of elements in $\widetilde{A}^0\otimes_{\k} \widetilde{A}^0$ which generate $I$. By Theorem \ref{thm-rgamma}, there is a functorial isomorphism
\[
\mrm{R}\Gamma_I( \mrm{L}\Lambda_{\a}(M) \otimes^{\mrm{L}}_{\k} N) \cong 
\opn{K}^{\vee}_{\infty}(\widetilde{A}^0\otimes_{\k} \widetilde{A}^0; \widetilde{\mathbf{a}}) \otimes^{\mrm{L}}_{\widetilde{A}^0\otimes_{\k} \widetilde{A}^0} ( \mrm{L}\Lambda_{\a}(M) \otimes^{\mrm{L}}_{\k} N).
\]
Letting $\mathbf{a}$ be the image of $\widetilde{\mathbf{a}}$ in $A$, we see that $\mathbf{a}$ generates $\a$, so using the fact that 
\[
\opn{K}^{\vee}_{\infty}(\widetilde{A}^0\otimes_{\k} \widetilde{A}^0; \widetilde{\mathbf{a}}) \otimes_{\widetilde{A}^0\otimes_{\k} \widetilde{A}^0} A \cong \opn{K}^{\vee}_{\infty}(A;\mathbf{a})
\]
and by Theorem \ref{thm-ring-mgm}, we get the required functorial isomorphism. The second isomorphism is proved similarly.
\end{proof}

\begin{cor}\label{cor-domain-of-def}
Let $\k$ be a noetherian ring. Let $(A,\a)$ be an adic ring which is an essentially formally of finite type preadic $\k$-algebra. Then there are isomorphisms
\[
\mrm{L}\Lambda_{\a} \mrm{R}\opn{Hom}_{A\otimes^{\mrm{L}}_{\k} A}(A,-\otimes^{\mrm{L}}_{\k} -) \cong 
\mrm{L}\Lambda_{\a} \mrm{R}\opn{Hom}_{A\otimes^{\mrm{L}}_{\k} A}(A,\mrm{L}\Lambda_{\a}(-) \otimes^{\mrm{L}}_{\k} \mrm{L}\Lambda_{\a}(-))
\]
and 
\[
\mrm{L}\Lambda_{\a} \mrm{R}\opn{Hom}_{A\otimes^{\mrm{L}}_{\k} A}(A,-\otimes^{\mrm{L}}_{\k} -) \cong 
\mrm{L}\Lambda_{\a} \mrm{R}\opn{Hom}_{A\otimes^{\mrm{L}}_{\k} A}(A,\mrm{R}\Gamma_{\a}(-) \otimes^{\mrm{L}}_{\k} \mrm{R}\Gamma_{\a}(-))
\]
of functors
\[
\mrm{D}(\opn{Mod} A) \times \mrm{D}(\opn{Mod} A) \to \mrm{D}(\opn{Mod} A)_{\opn{\a-com}}.
\]
\end{cor}
\begin{proof}
This follows immediately from Corollary \ref{cor-cformula-noetherian} and the previous lemma, using the symmetry of the derived tensor product.
\end{proof}

This corollary shows that the complete derived Hochschild cohomology functor depends only on the images of the given complexes in $\mrm{D}(\opn{Mod} A)_{\opn{\a-com}}$, so one should focus on the functor
\[
\mrm{L}\Lambda_{\a} \mrm{R}\opn{Hom}_{A\otimes^{\mrm{L}}_{\k} A}(A,-\otimes^{\mrm{L}}_{\k} -):
\mrm{D}(\opn{Mod} A)_{\opn{\a-com}} \times \mrm{D}(\opn{Mod} A)_{\opn{\a-com}} \to \mrm{D}(\opn{Mod} A)_{\opn{\a-com}}.
\]

We now turn to study the torsion case. In this case our results are weaker then in the complete case, but are probably sufficient for most noetherian applications.

\begin{lem}\label{lem-gammaofhom}
Let $(B,\b)$ be an internally flat preadic DG-algebra. Let $(C,\c)$ be a weakly proregular preadic ring, and let $B\to C$ be a preadic DG-algebra map. Assume that $H^0(B)$ is a noetherian ring, that for each $i<0$, $H^i(B)$ is a finitely generated $H^0(B)$-module, and that $C$ is a finitely generated $H^0(B)$-module. Suppose that $\b \cdot C = \c$. Then for any DG $B$-module $M$ which has a bounded below cohomology, there is a functorial isomorphism
\[
\mrm{R}\Gamma_{\c} \mrm{R}\opn{Hom}_B(C,M) \cong \mrm{R}\opn{Hom}_B(C,\mrm{R}\Gamma_{\b}(M))
\]
in $\mrm{D}(\opn{Mod} C)$.
\end{lem}
\begin{proof}
Let $\mathbf{b}$ be a finite sequence in $B^0$ that generates $\b$, and let $\mathbf{c}$ be its image in $C$. Since $\c$ is weakly proregular, by Theorem \ref{thm-ring-mgm}, there is a functorial isomorphism
\[
\mrm{R}\Gamma_{\c} \mrm{R}\opn{Hom}_B(C,M) \cong \opn{Tel}(C;\mathbf{c}) \otimes_{C}  \mrm{R}\opn{Hom}_B(C,M)
\]
in $\mrm{D}(\opn{Mod} C)$.
Since $\mathbf{b}C = \mathbf{c}$, there is an isomorphism of complexes
\[
\opn{Tel}(C;\mathbf{c}) \cong \opn{Tel}(B^0;\mathbf{b}) \otimes_{B^0} C.
\]
Hence, there is a functorial isomorphism in $\mrm{D}(\opn{Mod} C)$:
\[
\mrm{R}\Gamma_{\c} \mrm{R}\opn{Hom}_B(C,M) \cong 
\opn{Tel}(B^0;\mathbf{b}) \otimes_{B^0} \mrm{R}\opn{Hom}_B(C,M).
\]

The conditions of the lemma ensures that all conditions of \cite{YZ1}, Proposition 1.12(b) are satisfied. Hence, according to that proposition, the functorial morphism
\[
\opn{Tel}(B^0;\mathbf{b}) \otimes_{B^0} \mrm{R}\opn{Hom}_B(C,M) \to 
\mrm{R}\opn{Hom}_B(C, \opn{Tel}(B^0;\mathbf{b}) \otimes_{B^0} M)
\]
is an isomorphism. The lemma now follows from Theorem \ref{thm-rgamma}.
\end{proof}

\begin{prop}
Let $(B,\b)$ be an internally projective preadic DG-algebra, such that for each $i<0$, the $B^0$-module $B^i$ is projective. Assume that the ring $\Lambda_{\b}(B^0)$ is noetherian. Then the adic DG-algebra $(\Lambda_{\b}(B),\Lambda_{\b}(B)\cdot \b)$ is internally flat.
\end{prop}
\begin{proof}
This follows immediately from Proposition \ref{prop-completion-of-projective-is-flat}.
\end{proof}

\begin{lem}\label{lem-rgammahat-llambda}
Let $(B,\b)$ be an internally projective weakly proregular preadic DG-algebra, such that for each $i<0$, the $B^0$-module $B^i$ is projective. Suppose that the ring $\Lambda_{\b}(B^0)$ is noetherian. Let $\widehat{\b} = \b \cdot \Lambda_{\b}(B)$. Then there is an isomorphism 
\[
\mrm{R}\Gamma_{\widehat{\b}} \circ \mrm{L}\Lambda_{\b} (-) \cong \mrm{R}\Gamma_{\b} (-)
\]
of functors
\[
\widetilde{\mrm{D}}(\opn{DGMod} B) \to \widetilde{\mrm{D}}(\opn{DGMod} \Lambda_{\b}(B)).
\]
\end{lem}
\begin{proof}
By the previous proposition, the DG-algebra $\widehat{B} := \Lambda_{\b}(B)$ is internally flat. Since $\widehat{B}^0$ is noetherian, it follows that the ideal $\widehat{\b}$ is weakly proregular.

Let $M\in \widetilde{\mrm{D}}(\opn{DGMod} B)$ be a DG-module.
Let $P \cong M$ be a K-flat resolution over $B$. 
By Proposition \ref{prop-complete-rgamma}, there is a functorial isomorphism
\[
\mrm{R}\Gamma_{\b} (M) \cong \opn{K}^{\vee}_{\infty}(\widehat{B}; \widehat{\mathbf{b}}) \otimes_B P
\]
in $\widetilde{\mrm{D}}(\opn{DGMod} \Lambda_{\b}(B))$.

On the other hand, by Theorem \ref{thm-rgamma}, there is a functorial isomorphism
\[
\mrm{R}\Gamma_{\widehat{\b}} \circ \mrm{L}\Lambda_{\b} (M) \cong 
\opn{K}^{\vee}_{\infty}(\widehat{B}; \widehat{\mathbf{b}}) \otimes_{\widehat{B}} \widehat{P} \cong \opn{K}^{\vee}_{\infty}(B; \mathbf{b}) \otimes_{B} \widehat{P} 
\]
in $\widetilde{\mrm{D}}(\opn{DGMod} \Lambda_{\b}(B))$.

Notice that the $\widehat{B}$-DG module $\opn{K}^{\vee}_{\infty}(B; \mathbf{b}) \otimes_{B} \widehat{P}$ satisfies the assumptions of Lemma \ref{lem-tortens}, so that it is functorially isomorphic to 
\[
(\opn{K}^{\vee}_{\infty}(B; \mathbf{b}) \otimes_{B} \widehat{P} ) \otimes^{\mrm{L}}_B \widehat{B}.
\]
By Theorems \ref{thm-rgamma} and \ref{thm-RgammaLLambda}, and the fact that $\widehat{P} \cong \mrm{L}\Lambda_{\b}(P)$, we see that there is a $B$-linear functorial isomorphism
\[
\opn{K}^{\vee}_{\infty}(B; \mathbf{b}) \otimes_{B} \widehat{P} \cong 
\opn{K}^{\vee}_{\infty}(B; \mathbf{b}) \otimes_{B} P.
\]
Hence, there is a $\widehat{B}$-linear functorial isomorphism
\[
(\opn{K}^{\vee}_{\infty}(B; \mathbf{b}) \otimes_{B} \widehat{P} ) \otimes^{\mrm{L}}_B \widehat{B} \cong 
(\opn{K}^{\vee}_{\infty}(B; \mathbf{b}) \otimes_{B} P) \otimes^{\mrm{L}}_B \widehat{B} \cong \opn{K}^{\vee}_{\infty}(\widehat{B}; \widehat{\mathbf{b}}) \otimes_B P.
\]
This proves the result.
\end{proof}

Here is the main result of this section in the torsion case.
\begin{thm}\label{thm-main-t-formula}
Let $\k$ be a commutative ring. Let $(A,\a)$ be an adic ring which is a weakly proregular preadic $\k$-algebra. Let $\k \to \widetilde{A} \to A$ be a K-flat resolution of $\k \to A$ which is internally projective, such that for each $i<0$, the $\widetilde{A}^0$-module $\widetilde{A}^i$ is projective. Let $I\subseteq (\widetilde{A}\otimes_{\k} \widetilde{A})^0$ be a weakly proregular ideal, such that its image under the 
composed map 
\[
(\widetilde{A}\otimes_{\k} \widetilde{A})^0 \to \widetilde{A}^0 \to A
\]
is equal to $\a$. Assume that the ring $\Lambda_I(\widetilde{A}\otimes_{\k} \widetilde{A})^0$ is noetherian, and that for each $i<0$, the  $H^0(\Lambda_I(\widetilde{A}\otimes_{\k} \widetilde{A}))$-module
$H^i(\Lambda_I(\widetilde{A}\otimes_{\k} \widetilde{A}))$ is finitely generated.
Then there is an isomorphism 
\[
u^{\widetilde{A},I}:
\mrm{R}\Gamma_{\a} (\mrm{R}\opn{Hom}_{A\otimes^{\mrm{L}}_{\k} A} ( A, -\otimes^{\mrm{L}}_{\k} -) )  
\cong 
\mrm{R}\opn{Hom}_{\Lambda_I(\widetilde{A}\otimes_{\k} \widetilde{A})} (A, \mrm{R}\Gamma_I ( - \otimes^{\mrm{L}}_{\k} - ) )
\]
of functors
\[
\mrm{D}(\opn{Mod} A)_{\opn{f.fd(\k)}} \times \mrm{D}(\opn{Mod} A)_{\opn{f.fd(\k)}} \to \mrm{D}(\opn{Mod} A)_{\opn{\a-tor}}
\]
where $\mrm{D}(\opn{Mod} A)_{\opn{f.fd(\k)}}$ denotes the category of complexes over $A$ which are of finite flat dimension over $\k$.
\end{thm}
\begin{proof}
Notice that since $\widetilde{A}$ is internally projective, $\widetilde{A}\otimes_{\k} \widetilde{A}$ is also internally projective.
Let $M,N \in \mrm{D}(\opn{Mod} A)_{\opn{f.fd(\k)}}$. 
By Theorem \ref{thm-ring-mgm}, there is a functorial isomorphism in $\mrm{D}(\opn{Mod} A)$:
\[
\mrm{R}\Gamma_{\a} (\mrm{R}\opn{Hom}_{A\otimes^{\mrm{L}}_{\k} A} ( A, M\otimes^{\mrm{L}}_{\k} N) ) \cong 
\mrm{R}\Gamma_{\a} (\mrm{L}\Lambda_{\a} (\mrm{R}\opn{Hom}_{A\otimes^{\mrm{L}}_{\k} A} ( A, M\otimes^{\mrm{L}}_{\k} N)) )  
\]
By Theorem \ref{thm-main-c-formula}, there is a functorial isomorphism in $\mrm{D}(\opn{Mod} A)$:
\[
\mrm{R}\Gamma_{\a} (\mrm{L}\Lambda_{\a} (\mrm{R}\opn{Hom}_{A\otimes^{\mrm{L}}_{\k} A} ( A, M\otimes^{\mrm{L}}_{\k} N)) )  
\cong
\mrm{R}\Gamma_{\a}
(
\mrm{R}\opn{Hom}_{\Lambda_I(\widetilde{A}\otimes_{\k} \widetilde{A})} ( A, \mrm{L}\Lambda_I(M \otimes^{\mrm{L}}_{\k} N)  )
).
\]

Let $\widehat{I} = I\cdot \Lambda_I(\widetilde{A}^0 \otimes_{\k} \widetilde{A}^0)$. Since $M$ and $N$ have finite flat dimension over $\k$, it follows that the complex $M\otimes_{\k} N$ has bounded cohomologies. Hence, since $I$ is weakly proregular, Theorem \ref{thm-llambda} shows that the complex $\mrm{L}\Lambda_I(M \otimes^{\mrm{L}}_{\k} N)$ also has bounded cohomology. Thus, the conditions of Lemma \ref{lem-gammaofhom} are satisfied, so there is an $A$-linear functorial isomorphism
\[
\mrm{R}\Gamma_{\a}
(
\mrm{R}\opn{Hom}_{\Lambda_I(\widetilde{A}\otimes_{\k} \widetilde{A})} ( A, \mrm{L}\Lambda_I(M \otimes^{\mrm{L}}_{\k} N)  )
) \cong 
\mrm{R}\opn{Hom}_{\Lambda_I(\widetilde{A}\otimes_{\k} \widetilde{A})} ( A, 
\mrm{R}\Gamma_{\widehat{I}} (
\mrm{L}\Lambda_I(M \otimes^{\mrm{L}}_{\k} N) ) ).
\]
The result now follows from Lemma \ref{lem-rgammahat-llambda}.
\end{proof}

\begin{rem}
Again, as in Corollary \ref{cor-ind-c}, one may focus on the functor in the right hand side of the above theorem, and formulate an independence result as in that Corollary.
\end{rem}

\begin{rem}
If $\k$ is a noetherian ring, and $(A,\a)$ is an adic ring which is an essentially formally of finite type $\k$-algebra, then the resolution constructed in Proposition \ref{prop-existence-of-res}(2) satisfies the assumptions of the theorem.
\end{rem}

\section{Derived adic Hochschild homology}

In this short and final section we briefly discuss adic versions of the derived Hochschild homology functors. 

Again, we begin by recalling the construction in the discrete case from \cite{AILN}. Let $\k$ be a commutative ring, and let $A$ be commutative $\k$-algebra. According to \cite{AILN}, Theorem 3.9, if $\k \to \widetilde{A} \to A$ is a K-flat DG-resolution of $\k \to A$, there is a functor
\[
A \otimes^{\mrm{L}}_{\widetilde{A}\otimes_{\k} \widetilde{A}} \mrm{R}\opn{Hom}_{\k}(-,-) : \mrm{D}(\opn{Mod} A) \times \mrm{D}(\opn{Mod} A) \to \mrm{D}(\opn{Mod} A).
\]
Again, this construction is independent of the choice of a K-flat resolution of $\k \to A$, so one gets a functor denoted by
\[
A \otimes^{\mrm{L}}_{A\otimes^{\mrm{L}}_{\k} A} \mrm{R}\opn{Hom}_{\k}(-,-).
\]
As in the previous section, we look at the derived completion and derived torsion of this functor. 

\begin{dfn}
Let $\k$ be a commutative ring. Let $(A,\a)$ be a weakly proregular preadic $\k$-algebra.

\begin{enumerate}
\item The complete derived Hochschild homology functor is defined to be the functor:
\[
\mrm{L}\Lambda_{\a} (A \otimes^{\mrm{L}}_{A\otimes^{\mrm{L}}_{\k} A} \mrm{R}\opn{Hom}_{\k}(-,-) )  : \mrm{D}(\opn{Mod} A) \times \mrm{D}(\opn{Mod} A) \to \mrm{D}(\opn{Mod} A)_{\opn{\a-com}}
\]
\item The torsion derived Hochschild homology functor is defined to be the functor:
\[
\mrm{R}\Gamma_{\a} (A \otimes^{\mrm{L}}_{A\otimes^{\mrm{L}}_{\k} A} \mrm{R}\opn{Hom}_{\k}(-,-))  : \mrm{D}(\opn{Mod} A) \times \mrm{D}(\opn{Mod} A) \to \mrm{D}(\opn{Mod} A)_{\opn{\a-tor}}
\]
\end{enumerate}
\end{dfn}

\begin{thm}\label{thm-h-t-formula}
Let $\k$ be a commutative ring. Let $(A,\a)$ be an adic ring which is a weakly proregular preadic $\k$-algebra. Let $\k \to \widetilde{A} \to A$ be a K-flat resolution of $\k \to A$ which is internally flat. Let $I\subseteq (\widetilde{A}\otimes_{\k} \widetilde{A})^0$ be a weakly proregular ideal, such that its image under the 
composed map 
\[
(\widetilde{A}\otimes_{\k} \widetilde{A})^0 \to \widetilde{A}^0 \to A
\]
is equal to $\a$. Then there is an isomorphism 
\[
\mrm{R}\Gamma_{\a} (A \otimes^{\mrm{L}}_{A\otimes^{\mrm{L}}_{\k} A} \mrm{R}\opn{Hom}_{\k}(-,-)) \cong 
A \otimes^{\mrm{L}}_{\Lambda_I(\widetilde{A} \otimes_{\k} \widetilde{A})} \mrm{R}\Gamma_I(\mrm{R}\opn{Hom}_{\k}(-,-))
\]
of functors
\[
\mrm{D}(\opn{Mod} A) \times \mrm{D}(\opn{Mod} A) \to \mrm{D}(\opn{Mod} A)_{\opn{\a-tor}}.
\]
\end{thm}
\begin{proof}
Let $M,N \in \mrm{D}(\opn{Mod} A)$. By definition of the derived Hochschild homology functor, there is a functorial isomorphism in $\mrm{D}(\opn{Mod} A)_{\opn{\a-tor}}$:
\[
\mrm{R}\Gamma_{\a} (A \otimes^{\mrm{L}}_{A\otimes^{\mrm{L}}_{\k} A} \mrm{R}\opn{Hom}_{\k}(M,N)) \cong 
\mrm{R}\Gamma_{\a} (A \otimes^{\mrm{L}}_{\widetilde{A}\otimes_{\k} \widetilde{A}} \mrm{R}\opn{Hom}_{\k}(M,N)).
\]
Let $\widetilde{\mathbf{a}}$ be a finite sequence of elements in $\widetilde{A}^0\otimes_{\k} \widetilde{A}^0$ which generates $I$, and let $\mathbf{a}$ be its image in $A$. Since $\a$ is weakly proregular, by Theorem \ref{thm-ring-mgm}, and the fact that $I\cdot A = \a$, it follows that there is a functorial isomorphism
\[
\mrm{R}\Gamma_{\a} (A \otimes^{\mrm{L}}_{\widetilde{A}\otimes_{\k} \widetilde{A}} \mrm{R}\opn{Hom}_{\k}(M,N)) \cong 
\opn{K}^{\vee}_{\infty}(\widetilde{A}\otimes_{\k} \widetilde{A};\widetilde{\mathbf{a}}) \otimes_{\widetilde{A}\otimes_{\k} \widetilde{A}} (A \otimes^{\mrm{L}}_{\widetilde{A}\otimes_{\k} \widetilde{A}} \mrm{R}\opn{Hom}_{\k}(M,N))
\]
in $\mrm{D}(\opn{Mod} A)$.
By Theorem \ref{thm-rgamma}, we get a functorial isomorphism
\[
\opn{K}^{\vee}_{\infty}(\widetilde{A}\otimes_{\k} \widetilde{A};\widetilde{\mathbf{a}}) \otimes_{\widetilde{A}\otimes_{\k} \widetilde{A}} (A \otimes^{\mrm{L}}_{\widetilde{A}\otimes_{\k} \widetilde{A}} \mrm{R}\opn{Hom}_{\k}(M,N)) \cong 
A \otimes^{\mrm{L}}_{\widetilde{A}\otimes_{\k} \widetilde{A}}  \mrm{R}\Gamma_I(\mrm{R}\opn{Hom}_{\k}(M,N)).
\]
By associativity of the derived tensor product, and the fact that $A$ is $I$-adically complete, there is a functorial isomorphism
\[
A \otimes^{\mrm{L}}_{\widetilde{A}\otimes_{\k} \widetilde{A}}  \mrm{R}\Gamma_I(\mrm{R}\opn{Hom}_{\k}(M,N)) \cong 
A \otimes^{\mrm{L}}_{\Lambda_I(\widetilde{A}\otimes_{\k} \widetilde{A})}
(\Lambda_I(\widetilde{A}\otimes_{\k} \widetilde{A})\otimes^{\mrm{L}}_{\widetilde{A}\otimes_{\k} \widetilde{A}}  \mrm{R}\Gamma_I(\mrm{R}\opn{Hom}_{\k}(M,N))).
\]
The result now follows from Lemma \ref{lem-tortens}.
\end{proof}

\begin{rem}
In the case where $A$ is flat over $\k$, so that one may take $\widetilde{A} = A$ in the above theorem, and assuming that $I = \a \otimes_{\k} A + A\otimes_{\k} \a$, the functor $A\otimes^{\mrm{L}}_{\widehat{A\otimes_{\k} A}} -$ (more precisely, its cohomologies) of the right hand side of this theorem, was extensively studied in the book \cite{HU}. 
\end{rem}

\begin{rem}
In \cite{VdB2}, Remark 7.1.1, there is an example, due to Yekutieli, which shows that if $\k=\mathbb{C}$, and $A = \mathbb{C}[[t]]$, then the complex $A\otimes^{\mrm{L}}_{A\otimes_{\k} A} A$ is unbounded, while the complex
$A\otimes^{\mrm{L}}_{\widehat{A\otimes_{\k} A}} A$ is bounded. Thus, it appears that unlike the results of the previous section, Theorem \ref{thm-h-t-formula} does not have a counterpart in the complete case.
\end{rem}

\end{document}